\documentclass[english,twoside,11pt]{article}

\usepackage[english,ngerman]{babel}
\usepackage{amsmath,amsthm,amssymb,latexsym}
\usepackage{moreverb}
\usepackage{url}
\usepackage{graphicx, color}
\usepackage{longtable}
\usepackage{lscape}
\usepackage{booktabs}
\usepackage{tabularx}
\usepackage{psfrag}
\usepackage{multicol}
\usepackage{paralist}
\usepackage[font=small,labelfont=bf]{caption}
\usepackage[labelfont=rm, textfont=rm, labelformat=simple]{subcaption}
\usepackage{auto-pst-pdf}

\newtheorem{deff}{Definition}
\newtheorem{lem}[deff]{Lemma}

\newtheorem{thm}[deff]{Theorem}
\newtheorem{cor}[deff]{Corollary}

\newtheorem{prob}[deff]{Problem}
\newtheorem{rem}[deff]{Remark}

\newtheorem{mthm}{Main Theorem}
\newcommand\sd{\operatorname{sd}}

\makeatletter

\title{Extremal examples of collapsible complexes \\ and random discrete Morse theory}

\author{Karim A. Adiprasito\footnote{Supported by an EPDI postdoctoral fellowship and by the Romanian NASR,
CNCS --- UEFISCDI, project PN-II-ID-PCE-2011-3-0533.}\\%
{\small Hebrew U, Jerusalem, Israel} \\ \texttt{\footnotesize adiprasito@math.huji.ac.il}
  \and
Bruno Benedetti\footnote{Supported by the DFG Coll.\ Research Center TRR 109,
``Discretization in Geometry and Dynamics''
and NSF Grant 1600741, "`Geometric Combinatorics and Discrete Morse Theory}\phantom{i} \\
{\small U Miami, FL, USA} \\ \texttt{\footnotesize bruno@math.miami.edu}
  \and
Frank H.~Lutz\footnote{Supported by the DFG Research Group ``Polyhedral Surfaces'', 
by the DFG Coll.\ Research Center TRR 109 ``Discretization in Geometry and Dynamics'',
by \textsc{VILLUM FONDEN} through the Experimental Mathematics Network 
and by the Danish National Research Foundation (DNRF) through the Centre for Symmetry and Deformation.} \\
{\small TU Berlin, Germany} \\ \texttt{\footnotesize lutz@math.tu-berlin.de}
}

\date{}

\begin{document}

\selectlanguage{english}


\maketitle

\begin{abstract}
We present extremal constructions connected with the property of simplicial collapsibility.

(1) For each $d \ge 2$, there are collapsible (and shellable) simplicial $d$-complexes with only one free face. 
Also, there are non-evasive $d$-complexes with only two free faces. (Both results are optimal in all dimensions.)

(2) Optimal discrete Morse vectors need not be unique. We explicitly construct a contractible, but non-collapsible 
$3$-dimensional  simplicial complex with face vector $f=(106,596,1064,573)$ that admits two distinct optimal 
discrete Morse vectors, $(1,1,1,0)$ and $(1,0,1,1)$. Indeed, we show that in every dimension $d\geq 3$ there are 
contractible, non-collapsible simplicial $d$-complexes that have $(1,0,\dots,0,1,1,0)$ and $(1,0,\dots,0,0,1,1)$ 
as distinct optimal discrete Morse vectors.

(3) We give a first explicit example of a (non-PL) $5$-manifold, with face vector $f=(5013,72300,290944,$ $495912,383136,110880)$, 
that is collapsible but not homeomorphic to a ball.

Furthermore, we discuss possible improvements and drawbacks of random approaches to collapsibility and discrete Morse theory. 
We will introduce randomized versions  \texttt{random-lex-first} and \texttt{random-lex-last} of the \texttt{lex-first} and \texttt{lex-last} 
discrete Morse strategies of \cite{BenedettiLutz2014}, respectively --- and we will see that in many instances 
the \texttt{random-lex-last} strategy works significantly better than Benedetti--Lutz's (uniform) \texttt{random} strategy. 
On the theoretical side, we prove that after repeated barycentric subdivisions, the discrete Morse vectors found by randomized algorithms have, 
on average, an exponential (in the number of barycentric subdivisions) number of critical cells asymptotically almost surely.
\end{abstract}

\vspace{4mm}

\small

\section{Introduction}
Collapsibility was introduced by Whitehead in 1939 \cite{Whitehead1939}, as a ``simpler'' version of the topological notion of contractibility. 
Roughly speaking, collapsible simplicial complexes can be progressively retracted to a single vertex via some sequence of elementary 
combinatorial moves. Each of these moves reduces the size of the complex by deleting exactly two faces. The only requirements are 
that these two faces should be of consecutive dimension, and the larger of the two should be the unique face properly containing 
the smaller one (which is usually called ``\emph{free face}''). 

In dimension one, for example, the free faces of a graph are simply its leaves. Every tree can be reduced to a point by recursively deleting 
one leaf; thus all contractible $1$-complexes are collapsible.  In dimension $d \ge 2$, however, some contractible $d$-complexes have 
no free faces; therefore,  collapsible $d$-complexes are a proper subset of the contractible ones. Contractibility is not algorithmically 
decidable in general, cf.~\cite{Novikov}.  Collapsibility is, though the decision problem is NP-complete if $d \ge 3$~\cite{Tancer2012pre}. 
Relatively fast heuristic approaches with good practical behavior have been described in \cite{BenedettiLutz2014}. 

Here we present a few examples of complexes that are extremal with respect to the collapsibility property. 
The first one shows that one can find collapsible complexes where the beginning of any sequence of deletions is forced.

\begin{mthm}[Theorems~\ref{thm:Sigma} and~\ref{thm:non-evasive}]
For every $d \ge 2$, there are
\begin{compactenum}[\rm (i)]
\item collapsible (and even shellable) simplicial $d$-complexes with only $1$ free face, 
\item and non-evasive simplicial $d$-complexes with only $2$ free faces.
\end{compactenum}
\end{mthm}

\noindent Both results are optimal, compare Lemma \ref{lem:NEbound}. 

\medskip

Whitehead showed that all collapsible PL triangulations of $d$-manifolds are homeomorphic to the $d$-ball. ``PL'' stands 
for piecewise-linear and refers to the technical requirement that the closed star of every face should be itself piecewise-linearly 
homeomorphic to a standard ball; every smooth manifold admits PL triangulations (though in dimension $\ge 5$, it also admits non-PL ones). 
The PL assumption in Whitehead's theorem is really necessary: It is a consequence of Ancel--Guilbault's work, 
cf.\ also~\cite{AdiprasitoBenedetti2011pre}, that \emph{any} contractible manifold admits \emph{some }collapsible triangulation. 

A priori, finding an explicit description of a collapsible triangulation of a manifold different than a ball seems a hard challenge, for three reasons: 
\begin{compactenum}[(1)]
\item the size of (non-PL) triangulations of manifolds with non-trivial topology; 
\item the absence of effective upper bounds on the number of barycentric\footnote{By a result of Adiprasito--Benedetti \cite{AdiprasitoBenedetti2012apre},
         if some subdivision of a complex $C$ is collapsible, then also some \emph{iterated barycentric} subdivision of $C$ is collapsible.} subdivisions 
         needed to achieve collapsibility; 
\item the algorithmic difficulty of deciding collapsibility.
\end{compactenum}

Rather than bypassing these problems, we used a direct approach, and then relied on luck. Specifically, first we tried to realize one triangulation 
of a collapsible manifold different than a ball using as few simplices as possible. Once our efforts resulted in a triangulation with $1\,358\,186$ faces, 
we fed it to the heuristic algorithm \texttt{random-discrete-Morse} from \cite{BenedettiLutz2014}. To our surprise, the algorithm was able to digest 
this complex and show its collapsibility directly --- so that in fact no further barycentric subdivisions were necessary.

\begin{mthm}[Theorem~\ref{thm:non_5_ball}]
There is a simplicial $5$-dimensional manifold with face vector $f=(5013,72300,290944,$ $495912,383136,110880)$ that is collapsible 
but not homeomorphic to the $5$-ball.
\end{mthm}

\medskip

We then turn to an issue left open in Forman's discrete Morse theory \cite{Forman1998, Forman2002}. 
In this extension of Whitehead's theory, in order to progressively deconstruct a complex to a point, 
we are allowed to perform not only deletions of a free face (which leave the homotopy type unchanged), 
but also deletions of the interior of some top-dimensional simplex (which change the homotopy type in a very controlled way). 
The simplices deleted with steps of the second type are called ``\emph{critical cells}''.

Forman's approach works with complexes of arbitrary topology, but to detect such topology, one must keep track of the critical cells 
progressively deleted and of the way they were attached. A~coarse way to do this is to store in a vector the numbers $c_i$ of critical $i$-faces. 
The resulting vector $\mathbf{c}=(c_0, \ldots, c_d)$ is called \emph{a discrete Morse vector} of the complex. It depends on the homotopy type 
of the complex, and also on the particular sequence of deletions chosen. The discrete Morse vector does not determine the homotopy type 
of the complex (basically because it forgets about attaching maps). However, it yields useful information; for example, it provides 
upper bounds for the Betti numbers, or for the rank of the fundamental group.

The ``mission'' is thus to find sequences of deletions that keep the vector $\mathbf{c}=(c_0, \ldots, c_d)$ as small as possible. 
An \emph{optimal discrete Morse vector} is one that minimizes the number $|\mathbf{c}|= c_0 + \ldots + c_d$ of critical faces. 
At this point, one issue was left unclear, namely, whether the set of discrete Morse vectors admits a componentwise minimum, 
or if instead some complex might have more than one optimal discrete Morse vector. We answer this dilemma as follows. 

\begin{mthm}[Theorems~\ref{thm:two_3d} and~\ref{thm:two}]
For every $d\geq 3$ there is a contractible but non-collapsible $d$-dimensional simplicial complex that has 
two distinct optimal discrete Morse vectors $(1,0,\dots,0,1,1,0)$ and $(1,0,\dots,0,0,1,1)$ with three critical cells each.
\end{mthm}

\subsubsection*{Random models in discrete Morse theory and homology computations}
Discrete Morse theory for simplicial (or cubical) complexes is the basis for fast (co-)reduction techniques that are used 
in modern homology and persistent homology packages such as CHomP~\cite{Chomp}, RedHom~\cite{RedHom} and Perseus~\cite{Perseus}.
The main aim in these packages is to first reduce the size of some given complex before eventually setting up boundary matrices 
for (expensive) Smith normal form homology computations.

The problem of finding optimal discrete Morse functions is {\cal NP}-hard~\cite{JoswigPfetsch2006,LewinerLopesTavares2003a}. 
However, in practice one can find optimal (or close to optimal) discrete Morse vectors even for huge inputs~\cite{BenedettiLutz2014} --- and it is, 
on the contrary, non-trivial to construct explicit complicated triangulations of small size of spaces with trivial topology (such as balls or spheres) 
on which the search for good discrete Morse vectors produces poor results; see, for example, the constructions 
in~\cite{BenedettiLutz2014,Lutz2004b,LutzTsuruga2013ext}.

In an effort to measure how complicated a given triangulation is, a randomized approach to discrete Morse theory was introduced 
in \cite{BenedettiLutz2014}. In this approach, a given complex is deconstructed level-wise, working from the top dimension downwards. 
The output is a discrete Morse vector $(c_0, c_1, c_2, \ldots, c_d)$. The algorithm has two important features:

\begin{compactitem}
\item A  $k$-face is declared critical only if there are no free $(k-1)$-faces available, or in other words, only if we cannot go further 
         with collapsing steps. This tends to keep the number of critical faces to a minimum, and to speed up the algorithm. 
\item The decision of which free $(k-1)$-face should be collapsed away, or (if none) of which $k$-face should be removed as critical, i
         s performed uniformly at random. Randomness allows for a fair analysis of the triangulation, and leaves the door open 
         for relaunching the program multiple times, thus obtaining a whole spectrum of outputs, called  \emph{experienced discrete Morse spectrum}.
\end{compactitem}

The collection of discrete Morse vectors that could a priori be reached this way is called  \emph{discrete Morse spectrum}; see Definition~\ref{def:DMS}.

It was observed in \cite{BenedettiLutz2014} that the boundary spheres of many simplicial polytopes, even with a large number of vertices,
have an experienced (i.e., after, say, 10000 runs of the program) discrete Morse spectrum consisting of the sole discrete Morse vector $(1,0,\dots,0,1)$. 
In the case of ``complicated'' triangulations it even might happen that their barycentric subdivisions become ``easier'' in the sense that 
on average fewer critical cells are picked up in a random search (as we experienced for the examples \texttt{trefoil}, \texttt{double\_trefoil}, 
and \texttt{triple\_trefoil} in~\cite{BenedettiLutz2014}).
This, however, is the exception: 

\begin{mthm}[Theorem~\ref{thm:exp}]
Let $K$ be any simplicial complex of dimension $d \ge 3$. Then the random discrete Morse algorithm, applied to the $\ell$-th barycentric subdivision
$\mathrm{sd}^{\ell} K$, yields an expected number of\, $\Omega(e^{\ell})$ critical cells a.a.s.
\end{mthm}

It is known that simplicial polytopal $d$-spheres of dimension $d\geq 3$ can have a non-trivial discrete Morse spectrum 
\cite{BenedettiLutz2009pre,BenedettiLutz2014,KCrowleyEbinKahnReyfmanWhiteXue2003pre}.
Theorem~\ref{thm:exp} of Section~\ref{sec:asymptotic}  shows that the average discrete Morse vector
for the discrete Morse spectrum of a high barycentric subdivision of any simplicial $d$-complex, $d\geq 3$, 
becomes arbitrarily large; in particular,  for (polytopal) barycentric subdivisions of simplicial polytopal spheres.

Finally, the Appendix is devoted to a comparison of the \texttt{random-discrete-Morse} strategy introduced in \cite{BenedettiLutz2014}
and two new variations of it, the \texttt{random-lex-first} and the \texttt{random-lex-last} strategies,
which are based on the deterministic \texttt{lex-first} and \texttt{lex-last} strategies of \cite{BenedettiLutz2014}, respectively.
(The \texttt{lex-last} strategy was called ``rev\_lex'' in \cite{BenedettiLutz2014}. 
Here we use \texttt{lex-last} --- along with  \texttt{lex-first} --- to avoid confusion with (the term) reverse lexicographic ordering, which is different.)
In the \texttt{random-lex-first} and the \texttt{random-lex-last} strategies, we first randomly relabel the vertices of some given complex
and then perform all choices of free faces and critical faces by picking at every stage the lexicographically smallest and lexicographically 
largest faces as in the deterministic \texttt{lex-first} and \texttt{lex-last} strategies of \cite{BenedettiLutz2014}, respectively. 
The experiments we include in the Appendix indicate a surprising superior efficiency of \texttt{random-lex-last}.

\section{Collapsible complexes with fewest free faces}

For the definition of collapsible, nonevasive, and shellable, we refer the reader to \cite{BenedettiLutz2013}. 
Recall that for contractible complexes, shellable implies collapsible. A triangulation $B$ of a ball 
is called \emph{endocollapsible} if $B$ minus some facet (or equivalently, minus any facet)
collapses onto the boundary $\partial B$ \cite{BVB}. Any collapsible complex has at least one free face, 
namely, the one at which the sequence of elementary collapses starts. It is easy to see that every nonevasive complex 
always has at least two free faces (Lemma \ref{lem:NEbound}). The aim of this section is to construct collapsible 
resp.\ nonevasive complexes with exactly one resp.\ exactly two free faces. 
First we need to explain a geometric construction that is essentially known. 

\begin{deff}[Antiprism Subdivision] Let $P$ be an arbitrary polytope. The \emph{antiprism subdivision} $\mathfrak{A}(P)$ of $P$ is obtained as follows:
\begin{compactitem}
\item first we place in the interior of $P$ a constricted copy of $P^*$, the polar dual of $P$;
\item then we take the join of each face $\sigma$ of $P$ with the corresponding face $\sigma^*$ of $P^*$.
\end{compactitem}
\end{deff}

\begin{figure}[tb]
\begin{center}
\includegraphics[width=3.5cm]{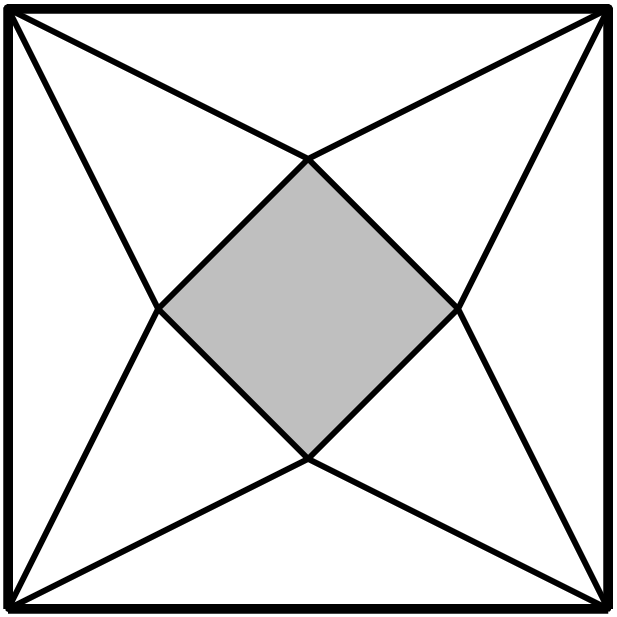}
\end{center}
\caption{The antiprism subdivision of a square.}
\label{fig:antiprism}
\end{figure}

This $\mathfrak{A}(P)$ is a polytopal complex, whose facets are in bijection with the proper faces of $P$ 
plus the facet $P^*$ in the interior; see Figure~\ref{fig:antiprism} for the antiprism subdivision of a square.
If~$P$ is not a simplex, $\mathfrak{A}(P)$ is not a simplicial complex. The next lemma shows how to canonically subdivide $\mathfrak{A}(P)$ 
to a simplicial complex that is always shellable.  A (pure) $d$-dimensional polytopal complex is \emph{regular} if it can be realized 
as the lower convex hull of a $(d+1)$-dimensional polytope. 
For the definition of  a \emph{pulling triangulation} see e.~g.\ \cite[Chapter 4.3.2]{LRS}.

\begin{lem}\label{lem:Antiprism}
If $P$ is a simplicial polytope, then there is a triangulation of $\mathfrak{A}(P)$ that is regular, and in particular shellable and endocollapsible.
\end{lem}

\begin{proof} 
If $P$ is $d$-dimensional, we can place $\mathfrak{A}(P)$ in a $d$-dimensional hyperplane of ${\mathbb R}^{d+1}$
and then lift $P^*$ to a parallel hyperplane. This way, we have realized $\mathfrak{A}(P)$ as the lower convex hull
of a $(d+1)$-polytope, i.e., $\mathfrak{A}(P)$ is regular. Next, we successively triangulate the non-simplicial faces
of $\mathfrak{A}(P)$ by picking vertices (one at a time) of $P^*$.
Let $v$ be such a vertex. We apply a pulling triangulation on the current subdivision of  $\mathfrak{A}(P)$
with respect to $v$. In the resulting subdivision of $\mathfrak{A}(P)$ in ${\mathbb R}^{d+1}$, $v$ lies in simplicial faces only, 
and we thus can slightly push $v$ ``outside'' to again obtain a regular subdivision of $\mathfrak{A}(P)$. 
We then proceed with further vertices of $P^*$ until we eventually obtain a simplicial and regular subdivision of $\mathfrak{A}(P)$.

Finally, every regular triangulation is shellable \cite{BM}, and every shellable triangulation is endocollapsible~\cite{BVB}.
\end{proof}

With Lemma~\ref{lem:Antiprism} we are now ready to construct a series of collapsible complexes with one free face only. 
The main idea is to start from the cross-polytope, stack all of its facets but one, triangulate its interior using a regular triangulation 
of the antiprism subdivision, and finally perform identifications on the boundary until ``only one facet is left''
--- thus maintaining the contractibility of the complex, and so the collapsibility. 

\begin{thm}\label{thm:Sigma}
For every $d\geq 2$ there is a collapsible shellable $d$-dimensional simplicial complex $\Sigma_d$ with $2^d+d+1$ vertices
that has only one free face.
\end{thm}

\begin{proof}
The proof consists of three parts: (i) we construct a contractible $d$-dimensional CW complex $S_d$ with only one free 
$(d-1)$-dimensional cell; (ii) we subdivide $S_d$ appropriately to obtain a simplicial complex $\Sigma_d$ 
with $2^d+d+1$ vertices and only one free face; and
(iii) we show via Lemma \ref{lem:Antiprism} that $\Sigma_d$ is collapsible and even shellable.

\begin{compactenum}[(i)]
\item Let $\mathrm{C}^{\Delta}_d$ denote the $d$-dimensional regular cross-polytope centered at the origin. 
Let us label the vertices of $\mathrm{C}^{\Delta}_d$ by $\{1,2,\dots,d,1',2',\dots,d'\}$ such that the antipodal map $x\mapsto -x$ 
maps each vertex $i$ to $i'$ and $i'$ to $i$, respectively. 
If $\{n_1,\dots,n_k,n_1',\dots, n_{d-k}'\}$ is a boundary facet of $\mathrm{C}^{\Delta}_d$ \emph{different} from $\{1,2,\dots,d\}$, 
we identify it with the facet $\{1',2',\dots,d'\}$ by mapping each $n_i$ to $n_i'$. 
This way, the $(d-1)$-dimensional face $\{1,2,\dots,d\}$ is the only boundary facet of $\mathrm{C}^{\Delta}_d$ 
that is not identified with another boundary facet of $\mathrm{C}^{\Delta}_d$.
We call the resulting cell complex $S_d$.

\item Let us go back to the cross polytope $\mathrm{C}^{\Delta}_d$ before the identification and subdivide the interior of $\mathrm{C}^{\Delta}_d$ 
according to an antiprism triangulation $T$. Formally, we place a $d$-dimensional cube $C_d$ with $2^d$ vertices, labeled by $\{d+2,\dots,2^d+d+1\}$,
in the interior of $\mathrm{C}^{\Delta}_d$, such that the cubical $(d-1)$-faces of the cube $C_d$ correspond to the vertices of $\mathrm{C}^{\Delta}_d$. 
The interior cube is then triangulated without adding further vertices, as explained in Lemma \ref{lem:Antiprism}, to achieve a 
regular triangulation $T$ of $\mathfrak{A}(\mathrm{C}^{\Delta}_d)$.

\begin{figure}[tb]
\begin{center}
\begin{postscript}
\psfrag{1}{1}
\psfrag{2}{2}
\psfrag{1'}{1'}
\psfrag{2'}{2'}
\psfrag{3}{3}
\psfrag{4}{4}
\psfrag{5}{5}
\psfrag{6}{6}
\psfrag{7}{7}
\includegraphics[width=5.5cm]{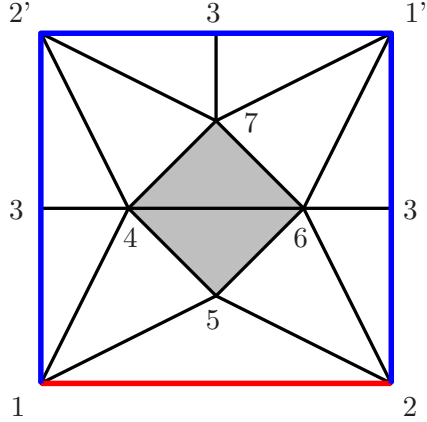}
\end{postscript}
\end{center}
\caption{The shellable complex $\Sigma_2$.}
\label{fig:sigma_2}
\end{figure}

\begin{figure}[tb]
\begin{center}
\begin{postscript}
\psfrag{1}{1}
\psfrag{2}{2}
\psfrag{3}{3}
\psfrag{4}{4}
\includegraphics[width=8cm]{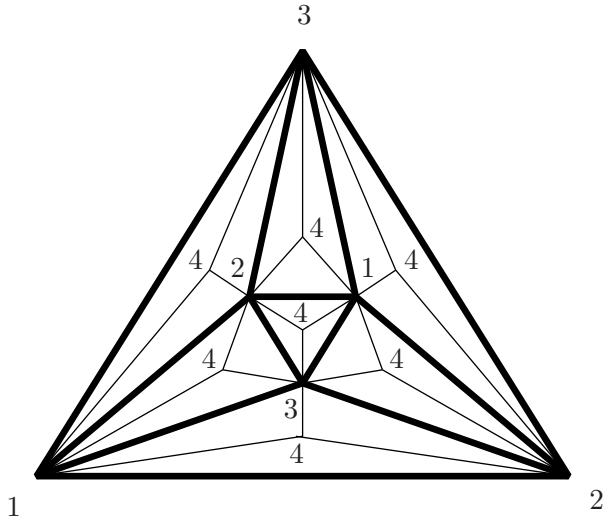}
\end{postscript}
\end{center}
\caption{The identified boundary of $\Sigma_3$.}
\label{fig:identified_stacked}
\end{figure}

Next, let us stack all the boundary facets of $\mathrm{C}^{\Delta}_d$ different from $\{1,2,\dots,d\}$. Under the identification 
of the (stacked) boundary facets, the boundary of $\mathrm{C}^{\Delta}_d$ gets mapped to the simplex $\{1,2,\dots,d\}$ union 
a subdivision of it --- thus to the boundary of a $d$-simplex with an extra vertex $d+1$ used for the stacking; see Figure~\ref{fig:sigma_2} 
for the complex $\Sigma_2$ with the free edge in red and the identified stacked boundary of $\mathrm{C}^{\Delta}_2$ in blue. 
Figure~\ref{fig:identified_stacked} displays (a ``Schlegel diagram'' of) the identified stacked boundary of $\mathrm{C}^{\Delta}_3$,
where the back triangle $\{1,2,3\}$ is not stacked.

We denote the resulting complex under the boundary identifications as explained in item (i)  by $\Sigma_d$. Unlike $S_d$, 
this $\Sigma_d$ is a simplicial complex. Further, we let $T'$ be the simplicial complex \emph{before} the boundary identifications, 
that is, $T'$ is obtained from the regular triangulation $T$ of $\mathfrak{A}(\mathrm{C}^{\Delta}_d)$ by stacking its boundary facets 
(the boundary facets of the outer $\mathrm{C}^{\Delta}_d$) different from $\{1,2,\dots,d\}$.

\item We embed the regular triangulation $T$ of $\mathfrak{A}(\mathrm{C}^{\Delta}_d)$ in ${\mathbb R}^{d+1}$ 
as the lower convex hull of a polytope $P^{d+1}$ such that the boundary facets of $T$ all lie in a $d$-dimensional hyperplane 
of ${\mathbb R}^{d+1}$ and bound $\mathrm{C}^{\Delta}_d$ as the single upper facet of $P^{d+1}$.
We shell the boundary complex of the polytope $P^{d+1}$ by using a line shelling that begins with the single upper facet $\mathrm{C}^{\Delta}_d$
and ends with the facet that contains (the not subdivided copy of) $\{1,2,\dots,d\}$ on the lower hull. 
In the reverse order, this shelling sequence deletes the boundary facets of $P^{d+1}$ one by one 
--- and induces a collapsing sequence for $T$ to its boundary minus $\{1,2,\dots,d\}$.

We can use this (reverse) shelling sequence to also induce a shelling sequence for $\Sigma_d$,  
where, before passing to $\Sigma_d$, we modify the sequence for $T$ locally to extend it to its subdivision~$T'$.
Since $T'$ is obtained from $T$ by stacking its boundary facets (except $\{1,2,\dots,d\}$), we merely 
have to replace those facets of $T$ in the shelling sequence that contain boundary facets different 
from $\{1,2,\dots,d\}$ by the $d$ resulting facets under the stackings, respectively.
In particular, we can shell $\Sigma_d$ in reverse order till we reach the cone over the subdivided facet $\{1',2',\dots,d'\}$,
which is shellable.
\qedhere
\end{compactenum}
\end{proof}

By construction, the complexes $\Sigma_d$ have $(d-1)$-dimensional faces that are contained in more than two facets.
In particular, the examples $\Sigma_d$ are not manifolds. For a shellable $3$-dimensional ball with only one ear see~\cite{Lutz_ear}.
None of the complexes $\Sigma_d$ is non-evasive, as the next lemma shows.

\begin{lem} \label{lem:NEbound}
A non-evasive $d$-complex, $d \ge 1$, has at least two free faces.  
\end{lem}

\begin{proof}
We proceed by induction on the dimension $d$. For $d =1$ the claim is equivalent to the well-known fact that every tree has at least two leaves. 
If $d \ge 2$, fix a non-evasive $d$-complex $C$ and a vertex $v$ of $C$ whose link and deletion are both non-evasive. By induction, 
the link of $v$ has two free faces $\sigma$ and $\tau$, which implies that $v \ast \sigma$ and $v \ast \tau$ are two free faces of $C$.
\end{proof}

As a consequence of Theorem \ref{thm:Sigma}, we show now that the bound given by Lemma \ref{lem:NEbound} is sharp in all dimensions:

\begin{figure}[tb]
\begin{center}
\begin{postscript}
\psfrag{v}{$v$}
\psfrag{w}{$w$}
\psfrag{sds2}{$\sd \Sigma_2$}
\psfrag{sds3}{$\sd \Sigma_3$}
\includegraphics[width=11cm]{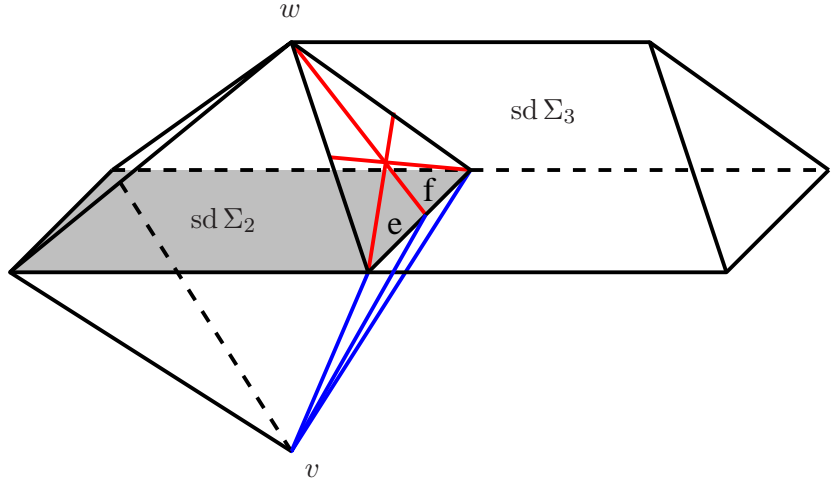}
\end{postscript}
\end{center}
\caption{The non-evasive complex $E_3$ with two free faces, $v \ast e$ and $v \ast f$.}
\label{fig:non-evasive_two_free}
\end{figure}

\begin{thm} \label{thm:non-evasive}
For every $d \ge 1$, there is a non-evasive $d$-complex $E_d$ with exactly two free faces, which share a codimension-one face. 
\end{thm}

\begin{proof}
The claim is obvious for $d=1$ (a path has two endpoints) and easy for $d=2$ (it suffices to consider the barycentric subdivision 
of the shellable $2$-complex $\Sigma_2$ with only one free face). In higher dimensions, one has to be more careful. 
We explain how to deal with dimension $3$, leaving it to the reader to extend our construction by induction on the dimension. 
Let us start with a non-evasive $2$-complex $E_2$ with two adjacent free edges, $e$ and $f$, say; 
the barycentric subdivision, $E_2=\sd \Sigma_2$, of $\Sigma_2$ yields such an example \cite{Welker1999}. 
The suspension \[S=\operatorname{susp} E_2 = (v \ast E_2) \cup (w \ast E_2)\] is also non-evasive, but it has four free triangles, 
pairwise adjacent. All we have to do is to ``kill the freeness'' of the triangles $w \ast e$ and $w \ast f$. Before proceeding with this, 
observe that $w \ast e$ and $w \ast f$ together can be thought of as one larger triangle $T=w\ast (e \cup f)$ stellarly subdivided into two. 
Since any barycentric subdivision can be obtained by a sequence of stellar subdivisions, there is a way to stellarly subdivide $S$ 
into a triangulation $S'$ that restricted to $(w \ast e) \cup (w \ast f)$ is combinatorially equivalent to the barycentric subdivision of $T$. 
Stellar subdivisions maintain non-evasiveness; so $S'$ is a non-evasive $3$-complex with $2 + 3!$ free triangles. 

Now, let $\Sigma_3$ be the collapsible $3$-complex with only one free face $\sigma$ constructed in Theorem \ref{thm:Sigma}. 
Let $\sd \Sigma_3$ be its barycentric subdivision. Let us glue $\sd \Sigma_3$ to $S'$ by identifying the subdivision 
of  $(w \ast e) \cup (w\ast f)$ with $\sd \sigma$. Let $E_3$ be the resulting complex; see Figure~\ref{fig:non-evasive_two_free}. 
$E_3$ has now only two free faces, $v \ast e$ and $v \ast f$. 

To show that $E_3$ is non-evasive, we have to give an order of the vertices of $E_3$ such that in every step the link and the deletion 
of the current vertex are non-evasive. Now, barycentric subdivisions of collapsible complexes and cones are non-evasive \cite{Welker1999}.
We start by deleting $v$. The link of $v$ is $E_2=\sd \Sigma_2$ and therefore non-evasive.
For our deletion sequence we proceed with the interior vertices of $E_2$ (in any order)
to collapse the subdivision of $w\ast E_2$ towards the subdivision of $(w \ast e) \cup (w \ast f)$. 
This leaves us with $\sd \Sigma_3$, which is non-evasive.
\end{proof}

\begin{rem}  \rm
A direct construction, though with more vertices, of non-evasive $d$-complexes $\overline{E}_d$ that have exactly two free faces 
is given by
$$\overline{E}_d=\sd(w\ast \Sigma_{d-1})\cup (v\ast \sd \Sigma_{d-1})\cup \sd \Sigma_d.$$
Here, the boundary of $\sd(w\ast \Sigma_{d-1})$ is $\partial(\sd(w\ast \Sigma_{d-1}))=\sd \Sigma_{d-1}\cup \sd T$,
to which we glue $v\ast \sd \Sigma_{d-1}$ (from ``below'' in Figure~\ref{fig:non-evasive_two_free}) 
and $\sd \Sigma_d$ (from ``the right'' in Figure~\ref{fig:non-evasive_two_free}).
\end{rem}

\section{Complexes with two different optimal Morse vectors} 
\label{sec:different_optima}

Here, we address the question of whether a simplicial complex must have a unique optimal discrete Morse vector. 
The answer is  negative in general, as we shall now see. 
We construct an explicit $3$-dimensional simplicial complex \texttt{two\_optima} with 106 vertices, which has two distinct optimal discrete Morse vectors 
$(1,1,1,0)$ and $(1,0,1,1)$. The construction of  \texttt{two\_optima} can be generalized to every dimension $d\geq 3$ to yield 
(up to suitable subdivisions) $d$-dimensional complexes with exactly two optimal discrete Morse vectors $(1,0,\dots,0,1,1,0)$ and $(1,0,\dots,0,0,1,1)$.
We also obtain that not all optimal discrete Morse vectors need to be contained in the discrete Morse spectrum of a complex.

The starting points for the construction of  \texttt{two\_optima} are the following observations: 
\begin{compactenum}[(i)]
\item for all discrete Morse vectors of a given simplicial complex $C$, the alternating sum $\sum_{i=0}^d(-1)^i c_i$ is always constant, 
         and equal to the Euler characteristic of $C$;
\item if $(c_0,c_1,\dots,c_d)$ is a discrete Morse vector for $C$, and $c_0$ is larger than the number of connected components of $C$, 
         then $(c_0 - 1,c_1 -1,\dots,c_d)$ is also a discrete Morse vector for $C$. (In particular, if we look for optimal discrete Morse vectors 
         of a connected complex, we can always assume $c_0=1$.)
\end{compactenum}
It follows that if we aim at producing a simplicial complex that has distinct optimal discrete Morse vectors, then $(1,1,1,0)$ and $(1,0,1,1)$ 
would be the smallest such vectors with respect to dimension and the total number of critical cells. By the discrete Morse inequalities
and the fact that any complex with discrete Morse vector $(1,c_1,c_2,c_3)$ is homotopy equivalent to a complex with the respective
number of cells in each dimension, any complex that admits both vectors $(1,1,1,0)$ and $(1,0,1,1)$  has to be contractible.

In the following, we indeed will construct such a complex \texttt{two\_optima}  that is contractible and has $(1,1,1,0)$ and $(1,0,1,1)$
as optimal discrete Morse vectors, where the main work in the construction will be to block the trivial vector $(1,0,0,0)$
to occur as discrete Morse vector of the complex.

\medskip
\textbf{Preliminary Example.} From the previous section we know that there is a $3$-complex $\Sigma_3$ that has only one free triangle $t$. 
Also, there is a $2$-complex $\Sigma_2$ with only one free edge $e$, which belongs to a triangle $t'$, say.  Let $C$ be the (non-pure) 
$3$-complex obtained by gluing together a copy of $\Sigma_2$ and a copy of $\Sigma_3$, via the identification $t \equiv t'$ (in some order of the vertices). 
Let us look for small discrete Morse vectors for $C$.
\begin{compactitem}
\item The vector $(1,0,1,1)$ can be achieved. In fact, since $\Sigma_3$ is endocollapsible, 
after removing a tetrahedron from $C$ we can collapse away all the tetrahedra of $C$; 
what we are left with is $\Sigma_2$ plus an additional $2$-dimensional membrane,
which, as the identified boundary of $\Sigma_3$, is the stellar subdivision of a triangle. 
Once we remove a critical triangle from this membrane, 
the resulting $2$-complex collapses to $\Sigma_2$, which is collapsible. 
Note that in this sequence, $e$ and $t'$ are removed together, in an elementary collapse. 
(Alternatively, we could have removed $t'$ as a critical triangle to obtain $\Sigma_2$ 
with the triangle~$t'$ stellarly subdivided, which is collapsible again, also with the then free edge $e$ paired up.)
\item Instead of removing a critical tetrahedron from $C$ as a first step, we could have started with an elementary collapse
         that uses the unique free triangle $t'$, and then empty out $\Sigma_3$ so that we obtain
         $\Sigma_2$ with the triangle~$t'$ subdivided, which is collapsible by starting with the free edge $e$.
         Thus, $(1,0,0,0)$ is achievable as a discrete Morse vector (and obviously smallest possible).
         Our aim will be to block the edge $e$ from being free and thus ruling out $(1,0,0,0)$ as possible discrete Morse vector --- with the hope 
         that both of the vectors $(1,1,1,0)$ and $(1,0,1,1)$  will become the new optima.
\end{compactitem}

\begin{thm}\label{thm:two_3d}
There is a contractible, but non-collapsible $3$-dimensional simplicial complex \linebreak
\texttt{two\_optima}  with face vector $f=(106,596,1064,573)$
that has two distinct optimal discrete Morse vectors, $(1,1,1,0)$ and $(1,0,1,1)$.
\end{thm}

\begin{proof}
The construction of  \texttt{two\_optima} involves a copy of $\Sigma_2$ (with relabeled vertices), 
a modified copy $\Sigma'_3$ of $\Sigma_3$ and nine further copies of $\Sigma_3$. 
The proof outline is as follows:
\begin{compactitem}[--]
\item In \textbf{Part I} we introduce the complex  $\Sigma_2 \cup \Sigma'_3$. The transition from $\Sigma_3$ to $\Sigma'_3$ consists in a subdivision followed by an identification of the edge $e$ with a segment of the identified boundary of $\Sigma_3$. As a result, $\Sigma_2 \cup \Sigma'_3$ is still contractible.
\item In \textbf{Part II} we study the small discrete Morse vectors of $\Sigma_2 \cup \Sigma'_3$. It turns out that the complex admits  $(1,0,1,1)$, $(1,1,1,0)$, 
but still $(1,0,0,0)$. However, the sequences of collapses yielding the third vector are now ``rare'', and all of a specific type. 
\item In \textbf{Part III} we show how  attaching nine further copies of $\Sigma_3$ to $\Sigma_2 \cup \Sigma'_3$ excludes collapsibility and yields the desired complex.
\end{compactitem}

\paragraph*{Part I.}
Let us start with a notational issue: We shall relabel the vertices of $\Sigma_2$ of Figure~\ref{fig:sigma_2} such that the unique free edge $1\,2$ 
is contained in the triangle $1\,2\,3$; see Figure~\ref{fig:sigma_2_relabeled}. To the complex $\Sigma_2 \cup \Sigma'_3$ we are going to construct,
(the relabeled copy of)  $\Sigma_2$ will thus contribute the 12 triangles
{\small
\[
\begin{array}{l@{\hspace{3.5mm}}l@{\hspace{3.5mm}}l@{\hspace{3.5mm}}l@{\hspace{3.5mm}}l@{\hspace{3.5mm}}l@{\hspace{3.5mm}}l@{\hspace{3.5mm}}l@{\hspace{3.5mm}}l@{\hspace{3.5mm}}l}

1\,3\,23,  &  1\,22\,23,  &  1\,22\,24,  &  1\,22\,25,  &  1\,24\,25,  & 2\,3\,25, \\
2\,22\,23,  & 2\,22\,24,  & 2\,22\,25,  & 2\,23\,24,  &  3\,23\,25,  & 23\,24\,25

\end{array}
\]
}%
as facets. On top of the 13th triangle $1\,2\,3$ (in grey) we glue the modified copy $\Sigma'_3$ of $\Sigma_3$.

\begin{figure}[tb]
\begin{center}
\begin{postscript}
\psfrag{1}{1}
\psfrag{2}{2}
\psfrag{3}{3}
\psfrag{22}{22}
\psfrag{23}{23}
\psfrag{24}{24}
\psfrag{25}{25}
\includegraphics[width=5.5cm]{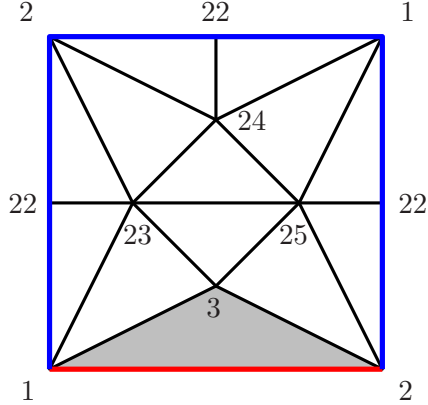}
\end{postscript}
\end{center}
\caption{The collapsible complex $\Sigma_2$ with new vertex labels.}
\label{fig:sigma_2_relabeled}
\end{figure}

\begin{figure}[tb]
\begin{center}
\begin{postscript}
\psfrag{1}{1}
\psfrag{2}{2}
\psfrag{3}{3}
\psfrag{4}{4}
\includegraphics[width=8cm]{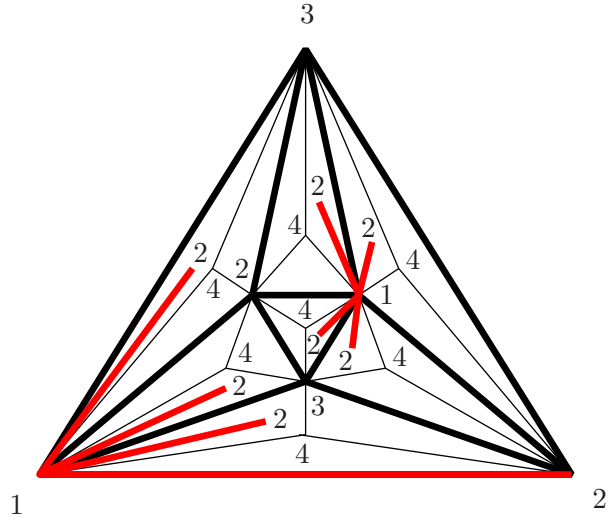}
\end{postscript}
\end{center}
\caption{The modified identified boundary of $\Sigma'_3$.}
\label{fig:identified_stacked_subdivided}
\end{figure}

In order to obtain $\Sigma'_3$ from $\Sigma_3$, we again start as in Figure~\ref{fig:identified_stacked} 
with the boundary of the octahedron (the $3$-dimensional crosspolytope), but this time with a finer subdivision 
and an additional identification on the (identified) boundary. Figure~\ref{fig:identified_stacked_subdivided}
displays this extra identification, where we glue the edge $1\,2$ to a segment $1\,2$  (in red) within (each of the seven duplicates of) 
the triangle $1\,3\,4$ (on the identified boundary of $\Sigma_3$ of Figure~\ref{fig:identified_stacked}). 
Since the edge $1\,2$ and its image intersect in the single  vertex 1,  $\Sigma'_3$ (after triangulating its $3$-dimensional ``interior'')
and $\Sigma_3$ are homotopy equivalent. To realize $\Sigma'_3$ as an explicit $3$-dimensional simplicial complex,
we first subdivide the triangle $1\,3\,4$ inside the stacked triangle $1\,2\,3$ to host the segment $1\,2$ as a proper edge;
see Figure~\ref{fig:subdivided_triangle}.

\begin{figure}[tb]
\begin{center}
\begin{postscript}
\psfrag{1}{1}
\psfrag{2}{2}
\psfrag{3}{3}
\psfrag{4}{4}
\psfrag{5}{5}
\psfrag{6}{6}
\includegraphics[width=9.75cm]{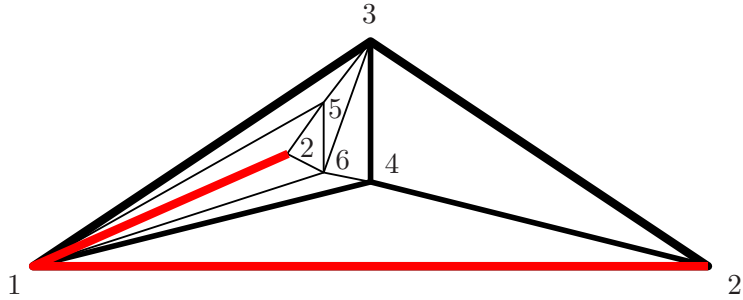}
\end{postscript}
\end{center}
\caption{The subdivision of the triangle $1\,3\,4$ within the triangle $1\,2\,3$.}
\label{fig:subdivided_triangle}
\end{figure}

We then glue the complex $\Sigma'_3$ on top of $\Sigma_2$ along the back-side triangle $1\,2\,3$  of $\Sigma'_3$.
We call $1\,2\,3$ the bottom triangle of $\Sigma'_3$, whereas on the top side of $\Sigma'_3$ we see
Figure~\ref{fig:identified_stacked_subdivided}, with each triangle $1\,3\,4$ subdivided as shown in Figure~\ref{fig:subdivided_triangle}.
The $3$-dimensional solid body between the top side and the back triangle can be thought of 
as a $3$-dimensional ball with identifications on the boundary. For triangulating this body,
we first shield off the extra seven copies of the vertex $2$ inside (the seven copies of) the triangle $1\,3\,4$
by taking seven cones 
{\small
\[
\begin{array}{l@{\hspace{3.5mm}}l@{\hspace{3.5mm}}l@{\hspace{3.5mm}}l@{\hspace{3.5mm}}l@{\hspace{3.5mm}}l@{\hspace{3.5mm}}l@{\hspace{3.5mm}}l@{\hspace{3.5mm}}l@{\hspace{3.5mm}}l@{\hspace{3.5mm}}l}

1\,2\,5\,x,   &  1\,2\,6\,x,   & 1\,3\,5\,x,  &  1\,4\,6\,x,   &   2\,5\,6\,x,  &  3\,4\,6\,x,  &  3\,5\,6\,x,

\end{array}
\]
}%
with respect to the apices $x=1,\dots,7$. Below these seven cones, we place a next layer of cones
{\small
\[
\begin{array}{l@{\hspace{3.5mm}}l@{\hspace{3.5mm}}l@{\hspace{3.5mm}}l@{\hspace{3.5mm}}l@{\hspace{3.5mm}}l@{\hspace{3.5mm}}l@{\hspace{3.5mm}}l@{\hspace{3.5mm}}l}

1\,3\,x\,y,  &
1\,4\,x\,y,  &
3\,4\,x\,y,  &

1\,2\,4\,x, &
2\,3\,4\,y,  &

\end{array}
\]
}%
with apices $y=x+7$ (so that the vertices $x=1,\dots,7$ are shielded from other copies of the vertex~2).

\begin{figure}[tb]
\begin{center}
\begin{postscript}
\psfrag{1}{1}
\psfrag{2}{2}
\psfrag{3}{3}
\psfrag{14}{14}
\psfrag{15}{15}
\psfrag{16}{16}
\psfrag{17}{17}
\psfrag{18}{18}
\psfrag{19}{19}
\psfrag{20}{20}
\includegraphics[width=7.25cm]{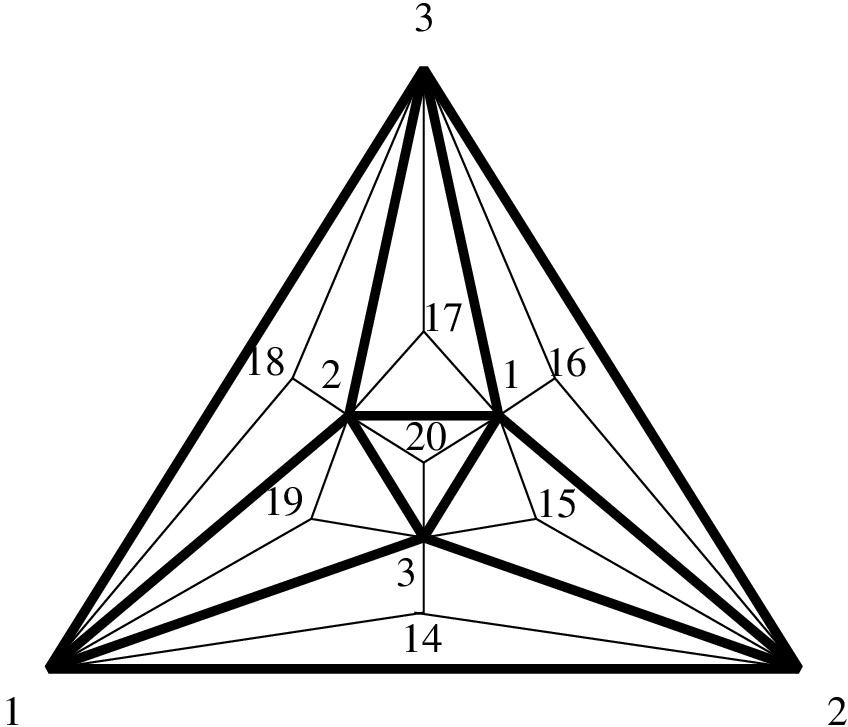}\hspace{5mm}\includegraphics[width=7.25cm]{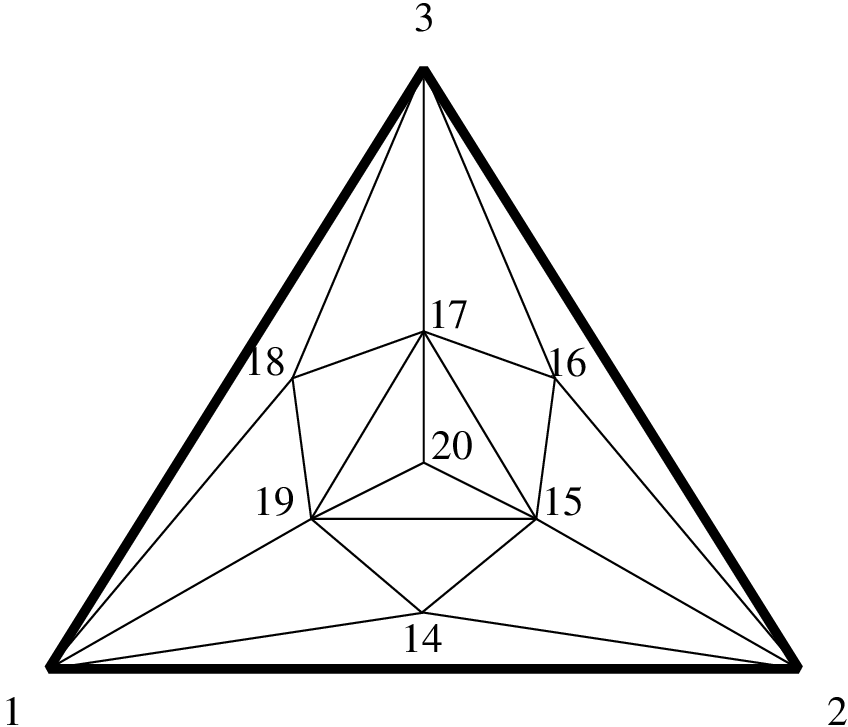}
\end{postscript}
\end{center}
\caption{Projection of the lower hull of the seven two-fold cones over the identified and subdivided boundary of the octahedron (left) 
              and one layer further below (right).}
\label{fig:identified_stacked_shielded}
\end{figure}

The seven two-fold cones can be assigned freely to the seven copies of the subdivided triangle $1\,2\,3$; 
our assignment is displayed in Figure~\ref{fig:identified_stacked_shielded} (left). We connect the cones
(and, en passant, shield off the multiple copies of the edges $1\,2$, $1\,3$ and $1\,4$ via the tetrahedra
{\small
\[
\begin{array}{l@{\hspace{3.5mm}}l@{\hspace{3.5mm}}l@{\hspace{3.5mm}}l@{\hspace{3.5mm}}l@{\hspace{3.5mm}}l}

2\,3\,14\,15,  &  1\,2\,15\,16,  &  1\,3\,16\,17,  &  2\,3\,17\,18,  &  1\,2\,18\,19,  &  1\,3\,14\,19,

\end{array}
\]
}%
and
{\small
\[
\begin{array}{l@{\hspace{3.5mm}}l@{\hspace{3.5mm}}l}

1\,3\,15\,20, &  1\,2\,17\,20,  &  2\,3\,19\,20.

\end{array}
\]
}%
Next, we shield off the central copies of the vertices $1$, $2$ and $3$
by gluing in the tetrahedra
{\small
\[
\begin{array}{l@{\hspace{3.5mm}}l@{\hspace{3.5mm}}l@{\hspace{3.5mm}}l@{\hspace{3.5mm}}l@{\hspace{3.5mm}}l@{\hspace{3.5mm}}l}

1\,15\,16\,17,  &
1\,15\,17\,20,  &

2\,17\,18\,19,  &
2\,17\,19\,20,  &

3\,14\,15\,19,  &
3\,15\,19\,20,  &

15\,17\,19\,20,  
\end{array}
\]
}%
and obtain as a lower envelope (i.e., boundary) of the above tetrahedra 13 triangles, which, together with the back triangle $1\,2\,3$,
give a triangulated $2$-sphere; see Figure~\ref{fig:identified_stacked_shielded} (right). We close this void by taking the cone 
over these $13+1$ triangles with respect to the additional vertex 21:
{\small
\[
\begin{array}{l@{\hspace{3.5mm}}l@{\hspace{3.5mm}}l@{\hspace{3.5mm}}l@{\hspace{3.5mm}}l@{\hspace{3.5mm}}l@{\hspace{3.5mm}}l}

1\,2\,14\,21, & 1\,3\,18\,21, & 1\,14\,19\,21, & 1\,18\,19\,21, & 2\,3\,16\,21, & 2\,14\,15\,21, & 2\,15\,16\,21, \\
3\,16\,17\,21, & 3\,17\,18\,21, & 14\,15\,19\,21, & 15\,16\,17\,21, & 15\,17\,19\,21, & 17\,18\,19\,21, \\[2mm]

    &&& 1\,2\,3\,21.
\end{array}
\]
}%

The resulting complex $\Sigma_2\cup\Sigma'_3$ has face vector $f=(25,128,218,114)$.

\paragraph*{Part II.}
We now describe the small discrete Morse vectors that can be obtained for the complex $\Sigma_2\cup\Sigma'_3$.
\begin{compactitem}
\item The vector $(1,1,1,0)$ is achievable. In fact, we can start with collapsing away 
all tetrahedra, where the only way to begin with is via the unique free triangle $1\,2\,3$.
Once we enter into the solid $3$-dimensional body of $\Sigma'_3$ we layer-wise collapse away the ``interior'', 
so that we are left with the identified (upper) boundary of $\Sigma'_3$. Under the identifications, the seven copies 
of the subdivided triangle $1\,2\,3$ fold up to just one copy, as displayed in Figure~\ref{fig:subdivided_triangle}, 
which, in turn, replaces the triangle $1\,2\,3$ in Figure~\ref{fig:sigma_2_relabeled}.
However, the once free edge $1\,2$ then is still glued to an interior edge and thus is not free. As a consequence,
a triangle, say, $1\,2\,4$ has to be marked as critical before we can continue with collapses. But then we are also forced
to pick up a critical edge, yielding $(1,1,1,0)$ as discrete Morse vector.

\item The vector $(1,0,1,1)$ can also be achieved. Let us initially remove a tetrahedron, say, $1\,2\,3\,21$ as critical 
and then empty out the interior of the solid $3$-dimensional body. We then are left with $\Sigma_2$ of Figure~\ref{fig:sigma_2_relabeled}
union the ``upper'' triangles of Figure~\ref{fig:subdivided_triangle}, which together is a space homotopy equivalent to the $2$-sphere.
If we remove one of the triangles of Figure~\ref{fig:subdivided_triangle} as critical, we can collapse down to $\Sigma_2$,
which is collapsible, thus resulting in $(1,0,1,1)$ as a discrete Morse vector.

\item The vectors $(1,1,1,0)$ and $(1,0,1,1)$ are not optimal for $\Sigma_2\cup\Sigma'_3$ --- in fact, $\Sigma_2\cup\Sigma'_3$ is collapsible
with optimal discrete Morse vector $(1,0,0,0)$. A respective collapsing is not obvious (we found some with the computer). 
It has to start via the unique free triangle $1\,2\,3$. Instead of emptying out the whole interior of the solid body and not touching
its identified boundary, we have to drill tunnels so that one of the nine triangles of Figure~\ref{fig:subdivided_triangle} becomes free.
Then we can perforate the membrane of triangles of Figure~\ref{fig:subdivided_triangle} and, eventually, 
free the edge $1\,2$.
\end{compactitem}

\paragraph*{Part III.}
In this part, we further modify  $\Sigma_2\cup\Sigma'_3$ to obtain a contractible space that is not collapsible, but still admits 
the vectors $(1,1,1,0)$ and $(1,0,1,1)$. To this end, we glue on top of each of the triangles
 $$t_1\,t_2\,t_3\in \{1\,2\,5, 1\,2\,6, 1\,2\,4, 1\,3\,5, 1\,4\,6, 2\,3\,4, 2\,5\,6, 3\,4\,6, 3\,5\,6\}$$
a copy of $\Sigma_3$ along its free face, thus contributing $12-3=9$ additional vertices for each copy.
The encoding of the copies of $\Sigma_3$ is as follows.
For $a=26+9j$, $j=0,\dots,8$, and $b=a+k$, $k=1\dots 7$,
we add the tetrahedra
{\small
\[
\begin{array}{l@{\hspace{3.5mm}}l@{\hspace{3.5mm}}l}
 t_1\,t_2\,a\,b,   &   t_1\,t_3\,a\,b,   &    t_2\,t_3\,a\,b  \\
\end{array}
\]
}%
to shield off the nine stacking vertices $a=26+9k$, $j=0,\dots,8$. We then glue in
the tetrahedra
{\small
\[
\begin{array}{l@{\hspace{3.5mm}}l}
t_2\,t_3\,a+1\,a+2,  &   t_1\,t_2\,a+2\,a+3,   \\ 
t_1\,t_3\,a+3\,a+4,   &  t_2\,t_3\,a+4\,a+5,   \\ 
t_1\,t_2\,a+5\,a+6,   &  t_1\,t_3\,a+6\,a+1
\end{array}
\]
}%
as well as
{\small
\[
\begin{array}{l@{\hspace{3.5mm}}l@{\hspace{3.5mm}}l}
t_1\,t_3\,a+2\,a+7,   &   t_1\,t_2\,a+4\,a+7,   &  t_2\,t_3\,a+6\,a+7 \\
\end{array}
\]
}%
to join the cones over the stacked triangles and then the tetrahedra
{\small
\[
\begin{array}{l@{\hspace{3.5mm}}l}
t_1\,a+2\,a+3\,a+4,   &  t_1\,a+2\,a+4\,a+7,  \\
t_2\,a+4\,a+5\,a+6,   &  t_2\,a+4\,a+6\,a+7,  \\
t_3\,a+1\,a+2\,a+6,   &  t_3\,a+2\,a+6\,a+7
\end{array}
\]
}%
over the ``boundary'' edges, with 
{\small
\[
\begin{array}{l@{\hspace{3.5mm}}l}
a+2\,a+4\,a+6\,a+7
\end{array}
\]
}%
as a final shielding piece. In a last step, we glue in the cone with respect to the vertex $a+8$ to fill in the central void:
{\small
\[
\begin{array}{l@{\hspace{5mm}}l@{\hspace{3.5mm}}l@{\hspace{5mm}}l}
t_1\,t_2\,t_3\,a+8,      \\
t_1\,t_2\,a+1\,a+8,    &   t_1\,t_3\,a+5\,a+8,     &  t_2\,t_3\,a+3\,a+8, \\
t_1\,a+1\,a+6\,a+8,   &   t_1\,a+5\,a+6\,a+8,   \\
t_2\,a+1\,a+2\,a+8,   &   t_2\,a+2\,a+3\,a+8,   \\
t_3\,a+3\,a+4\,a+8,    &  t_3\,a+4\,a+5\,a+8,  \\
a+1\,a+2\,a+6\,a+8,  &   a+2\,a+3\,a+4\,a+8,  &    a+4\,a+5\,a+6\,a+8,  \\
a+2\,a+4\,a+6\,a+8.
\end{array}
\]
}%
For simplicity, the resulting complex $\Sigma_2\cup\Sigma'_3\cup_{i=1,\dots,9}\Sigma_3$
will be called \texttt{two\_optima}; a list of facets of the complex is available online at \cite{BenedettiLutz_LIBRARY}.

To see that the  complex \texttt{two\_optima} admits the discrete Morse vectors $(1,1,1,0)$ and $(1,0,1,1)$ 
we can proceed as above:
\begin{compactitem}
\item For $(1,1,1,0)$, we enter via the unique free triangle $1\,2\,3$, empty out the interior of $\Sigma'_3$.
         At this point, we can enter the nine copies of $\Sigma_3$ via the freed triangles
         $1\,2\,5$, $1\,2\,6$, $1\,2\,4$, $1\,3\,5$, $1\,4\,6$, $2\,3\,4$, $2\,5\,6$, $3\,4\,6$, $3\,5\,6$ 
         and empty out the interiors of the copies. By construction, the edge $1\,2$ then is blocked, 
         but can be freed once we declare one of the boundary triangles of $\cup_{i=1,\dots,9}\Sigma_3$ to be critical.
\item For the vector $(1,0,1,1)$ we can proceed similarly, once we have marked as critical and 
         removed one of the interior tetrahedra of $\Sigma'_3$.
\end{compactitem}

Finally, we need to see that \texttt{two\_optima} is not collapsible.
Since each copy of $\Sigma_3$ can only be entered via one free triangle (say, $a\,b\,c$) that is glued 
on top of $\Sigma'_3$, we need to first empty out all tetrahedra of $\Sigma'_3$ that touch the triangle $a\,b\,c$
before we can enter the respective copy of $\Sigma_3$. At that point, the copy of $\Sigma_3$
is glued to the remainder via the three boundary edges $a\,b$, $a\,c$ and $b\,c$ of the triangle $a\,b\,c$.
These three edges stay fixed, in particular, cannot be free. Once we have collapsed the copy of $\Sigma_3$, 
we obtain a membrane, which is (homotopy equivalent to) a $2$-dimensional disc glued in along 
the boundary of the triangle $a\,b\,c$. It follows that none of the edges $a\,b$, $a\,c$ and $b\,c$
as well as none of the boundary edges of the other copies $\Sigma_3$ can become free,
\emph{before} we mark some triangle as critical and perforate the collection of $2$-dimensional
membranes this way.
\end{proof}

Our construction can be generalized to dimensions $d>3$.

\begin{lem}\label{lem:homot_disk}
Let $D$ be a $d$-disk, and let $\gamma$ be any $(d-2)$-loop in $\partial D$. Let $D' \subset D$ be any CW complex 
such that $D$ deformation retracts to $D'$ which contains $\gamma$. Then there is a $(d-1)$-disk $d' \hookrightarrow D'$ 
with $\partial d'\hookrightarrow \gamma$. 
\end{lem}

\begin{proof}
Choose a disk $d\in \partial D$ with $\partial d = \gamma$. Then the deformation retract $f_t: D\longrightarrow D'$ deforms $d$ to $d'$. 
\end{proof}

\begin{thm}\label{thm:two}
For every $d\geq 3$ there is a contractible, but non-collapsible simplicial $d$-complex
that has two distinct optimal discrete Morse vectors $(1,0,\dots,0,1,1,0)$ and $(1,0,\dots,0,0,1,1)$.
\end{thm}

\begin{proof}
As before in the $3$-dimensional case, for $d\geq 4$, we start with a copy of $\Sigma_{d-1}$ 
for which its free face is labeled $1\,2\,\dots\,d-1$ and is contained in the $(d-1)$-simplex
$1\,2\,\dots\,d-1\,d$. On top of the  $(d-1)$-simplex $1\,2\,\dots\,d-1\,d$ of $\Sigma_{d-1}$ we glue a modified copy
$\Sigma'_d$ of $\Sigma_d$ along the unique free face $1\,2\,\dots\,d-1\,d$ of $\Sigma'_d$. 
Here, the modified version $\Sigma'_d$ is obtained from $\Sigma_d$ be identifying  the face $1\,2\,\dots\,d-1$ 
with a $(d-2)$-dimensional disk on the identified boundary of $\Sigma_d$ such that the face $1\,2\,\dots\,d-1$
and its image intersect in the vertex $1$.
On top of the $(d-1)$-dimensional faces of the identified boundary of $\Sigma'_d$ (up to a suitable subdivision), 
we glue copies of $\Sigma_d$ to avoid (by Lemma~\ref{lem:homot_disk}) perforation.
\end{proof}

\begin{rem}  \rm
Each of the discrete Morse algorithms \texttt{random}, \texttt{random-lex-first}, \texttt{random-lex-last}
does not remove any $d$-face as critical as long as nontrivial collapses of $d$-faces are possible. 
Therefore, when applied to \texttt{two-optima}, none of the proposed algorithms can see a discrete Morse function 
with a critical $3$-cell --- and indeed, we only found $(1,1,1,0)$ in our computations and never picked up extra cells 
in the lower dimensions $0$, $1$, and $2$.
\end{rem}

In \cite{BenedettiLutz2014}, the \emph{discrete Morse spectrum} of a simplicial complex was defined
as the set of all possible outcomes of the random discrete Morse algorithm along with the 
respective probabilities for the individual vectors. We next give an alternative definition
in terms of monotone discrete Morse functions.

\begin{deff}[Monotone discrete Morse function]\label{def:montone}
A \emph{monotone discrete Morse function} on a simplicial complex $C$ is a map $f: C \rightarrow \mathbb{Z}$ 
satisfying the following six axioms:
\begin{compactenum}[\rm (i)]
\item if $\sigma \subseteq \tau$, then $f(\sigma) \le f(\tau)$;
\item the cardinality of $f^{-1} (q)$ is at most $2$ for each $q \in \mathbb{Z}$ ;
\item if $f(\sigma) = f(\tau)$, then either $\sigma \subseteq \tau$ or $\tau \subseteq \sigma$;
\item for any $\sigma \subseteq \tau$ and $\sigma' \subseteq \tau'$,
if $f(\sigma) = f(\tau) \le f(\sigma') =f(\tau')$ then 
$\dim \tau \le \dim \tau'$;
\item $f(C) = [0, M]$, for some $M \in \mathbb{N}$;
\item for any \emph{critical face} $\Delta$ (that is, a face such that $f(\sigma) \ne f(\Delta)$ for each face $\sigma \ne \Delta$), 
the complex $\{ \sigma \in C \textrm{ s.t. } f(\sigma) \le f(\Delta) \}$ has no free $(\dim \Delta -1)$-dimensional face.
\end{compactenum}
\end{deff}

\begin{deff}[Discrete Morse spectrum] \label{def:DMS}
Let $C$ be any simplicial complex. The \emph{discrete Morse spectrum} of $C$ is the set of all discrete Morse vectors 
of all monotone discrete Morse functions on~$C$.
\end{deff}

\begin{rem} \rm
Given an arbitrary discrete Morse function $f$ in the sense of Forman, it is easy to produce a function $g$ that induces 
the same Morse matching of $f$ and in addition satisfies (i), (ii), (iii), (iv) and (v). In fact, as explained by Forman \cite{Forman2002}, 
the function $f$ induces a step-by-step deconstruction of the complex: in the $i$-th step we either delete a critical face $\Delta_i$, 
or we delete a pair of faces $(\sigma_i, \Sigma_i)$, with $\sigma_i$ free. In the first case, let us set $\tilde{g} (\Delta_i)= -i$; 
in the second case, we set $\tilde{g} (\sigma_i)= \tilde{g}(\Sigma_i)=-i$. We leave it to the reader to verify that the resulting map $\tilde{g}$ 
satisfies axioms (i) to (iv). If $M = \min \{\tilde{g}(\sigma): \sigma \in C\}$, the function $g$ we seek is then defined by
$g (\sigma) = \tilde{g}(\sigma) - M$ for each $\sigma$.
\end{rem}

\begin{rem} \label{rem:restrictive}
\rm The sixth axiom of Definition~\ref{def:montone} is instead ``really restrictive''. If $C$ is any complex of dimension~$\ge 1$, 
the $f$-vector of $C$ is \emph{not} in the discrete Morse spectrum, though it is obtainable by some discrete Morse function 
in the sense of Forman (corresponding to the ``empty  matching'', which leaves all faces critical). The sixth axiom in fact forbids us 
to declare a $k$-face critical if there are still free $(k-1)$-faces available. 
\end{rem}

Any vector output from any of the discrete Morse strategies \texttt{random, random-lex-first, random-lex-last}, is by construction 
the discrete Morse vector of a \emph{monotone} discrete Morse function. Therefore, none of them will ever output a vector 
corresponding to the empty matching. However, this is not at all a drawback: As we said, the essence of Forman's theory 
is to come up with discrete Morse vectors that are \emph{as small as possible}. So it is actually time-saving if some very 
non-optimal matchings (like the empty one, which is the worst possible) are systematically neglected by the algorithm we are using. 

We observed experimentally that these algorithms often find \emph{optimal} matchings. This naturally triggers the question 
if the discrete Morse spectrum always contains an optimal vector. (If this is the case, then it might actually be more efficient, 
for all sorts of discrete-Morse-theoretic computations, to switch to the ``monotone discrete Morse function'' setup.)

What we can derive from Theorem \ref{thm:two} is that not \emph{all} optimal discrete Morse vectors belong to the spectrum:

\begin{cor} 
On the complex  \texttt{two\_optima} of Theorem \ref{thm:two}, no monotone discrete Morse function reaches the optimal Morse vector $(1,0,1,1)$.
\end{cor}

However, the vector $(1,1,1,0)$ is reached (in 10000 out of 10000 runs each) by all three strategies
\texttt{random}, \texttt{random-lex-first}, and \texttt{random-lex-last}.

\begin{prob}
Is at least one of the optimal discrete Morse vectors of a given simplicial complex reachable via a monotone discrete Morse function?
If so, does the lexicographically-largest among the optimal discrete Morse vectors always belong  to the spectrum?
\end{prob}

\section{A collapsible 5-manifold different from the 5-ball}

According to Whitehead \cite{Whitehead1939}, every collapsible combinatorial $d$-manifold 
is a combinatorial \mbox{$d$-ball}. 
On the other hand, every contractible $d$-manifold, if $d \ne 4$, admits a collapsible triangulation \cite{AdiprasitoBenedetti2011pre}; 
and in each dimension $d\geq 5$, there exist non-PL triangulations of contractible $d$-manifolds \emph{different} from $d$-balls.

In this section, our aim is to construct a first explicit ``small'' example of a collapsible non-PL triangulation of a contractible $5$-manifold 
different from the $5$-ball. Our construction is in seven steps and is based on ideas from \cite{AdiprasitoBenedetti2011pre}. 
In every step of the construction, we try to save on the size of the intermediate complexes.

\medskip

\noindent
\emph{1. Start with a (small) triangulation of a non-trivial homology $3$-sphere.}
      The smallest known triangulation of a non-trivial homology
      $3$-sphere is the $16$-vertex triangulation \texttt{poincare}~\cite{BjoernerLutz2000} 
      of the Poincar\'e homology $3$-sphere $\Sigma^3$. 
      In fact, the triangulation \texttt{poincare} with  
      $f=(16,106,180,90)$
      is conjectured to be the smallest triangulation of $\Sigma^3$ 
      and to be the unique triangulation of~$\Sigma^3$ with 
      this $f$-vector \cite[Conj.~6]{BjoernerLutz2000}.
      Triangulations of other non-trivial homology $3$-spheres 
      are believed to require more vertices and faces than \texttt{poincare}; cf.~\cite{LutzSulankeSwartz2009}.

\medskip
\smallskip

\noindent
\emph{2. Remove a (large) triangulated ball.}
      The triangulation \texttt{poincare} has vertex-valence-vector
      $\mbox{\rm val}=(14,14,11,14,14,11,14,12,13,15,15,14,15,15,15,6).$
      In particular, there are five vertices of valence $15$ 
      whose vertex-stars have $26$ tetrahedra each. 
      We remove the star of vertex $15$ from \texttt{poincare} and relabel
      vertex $16$ to $15$. The resulting $15$-vertex triangulation
      \texttt{poincare\_minus\_ball} has $64$ tetrahedra.

\medskip
\smallskip

\noindent
\emph{3. Take the cross product with an interval.}
      The cross product of \texttt{poincare\_minus\_ball} with an interval $I$ 
      is a prodsimplicial complex with $2\cdot 15=30$ vertices. Any prodsimplicial
      complex can be triangulated without adding new vertices; for product triangulations see \cite{Lutz2003bpre} 
      and references therein. This way, we obtain a $4$-dimensional simplicial complex
      \texttt{poincare\_minus\_ball\_x\_I}, which is a $30$-vertex triangulation
      of a $4$-manifold $K^4$ that has the connected sum $\Sigma^3\#\Sigma^3$ as its boundary.

\medskip
\smallskip

\noindent
\emph{4. Add cone over boundary.}
      If we add to $K^4$ the cone over its boundary with respect to 
      a new vertex~$31$, then the resulting $4$-dimensional
      simplicial complex $L^4$ is a combinatorial $4$-pseudomanifold,
      i.e., all of its vertex-links are combinatorial $3$-manifolds.
      In fact, the links of the vertices $1,\dots,30$ in $L^4$
      are triangulated $3$-spheres, whereas the link of vertex $31$ in $L^4$
      is a triangulation of the homology $3$-sphere $\Sigma^3\#\Sigma^3$.

\medskip
\smallskip

\noindent
\emph{5. Perform a one-point suspension.}
      By the double suspension theorem of Edwards \cite{Edwards1975} 
      and Cannon~\cite{Cannon1979}, the double suspension $\mbox{susp}(\mbox{susp}(H^d))$ 
      of any homology $d$-sphere $H^d$ is homeomorphic to the standard sphere $S^{d+2}$.
      In our case, $L^4$ is homeomorphic to the single suspension $\mbox{susp}(\Sigma^3\#\Sigma^3)$
      and therefore
       \[\mbox{susp}(L^4)\cong\mbox{susp}(\mbox{susp}(\Sigma^3\#\Sigma^3))\cong S^5.\]
              
      The triangulation $\mbox{susp}(L^4)$ of $S^5$ with $31+2$ vertices is non-PL, 
      since $\Sigma^3\#\Sigma^3$ occurs as the link of the edge 31--33
      (and of the edge 31--32) in $\mbox{susp}(L^4)$.
      Instead of taking the standard suspension $\mbox{susp}(L^4)$ with respect to two new
      vertices $32$ and $33$, the one-point suspension uses  
      one vertex less \cite{BjoernerLutz2000} --- for this,
      we simply contract the edge $31$--$33$ in $\mbox{susp}(L^4)$. 
      The resulting complex $\Sigma^5_{32}$ then has face-vector $f(\Sigma^5_{32})=(32,349,1352,2471,2154,718)$.

\medskip
\smallskip

\noindent
\emph{6. Subdivide barycentrically.}
      Let $\mbox{sd}(\Sigma^5_{32})$ be the barycentric subdivision of $\Sigma^5_{32}$
      with 
      \[f(\mbox{sd}(\Sigma^5_{32}))=(7076, 152540, 807888, 1696344, 1550880, 516960).\]
      This step is expensive, but paves the way for the final part of our construction.

\medskip
\smallskip

\noindent
\emph{7. Collar the PL singular set in\, $\mbox{\rm sd}(\Sigma^5_{32})$.}
      In the barycentric subdivision $\mbox{sd}(\Sigma^5_{32})$,
      there are precisely three vertices, $v_{\,31}$, $v_{\,31\mbox{--}32}$ and $v_{\,32}$,
      that do not have combinatorial $4$-spheres as their vertex-links
      --- these three vertices are connected by the edges 
      $e_{\,31,31\mbox{--}32}$ and $e_{\,31\mbox{--}32,32}$ 
      for which their links are triangulations of $\Sigma^3\#\Sigma^3$.
      The (contractible!) subcomplex of $\mbox{sd}(\Sigma^5_{32})$ formed by these two edges 
      is the \emph{PL singular set} of $\mbox{sd}(\Sigma^5_{32})$. 
      Let the collar of the contractible PL singular set in $\mbox{sd}(\Sigma^5_{32})$
      be the simplicial complex \texttt{contractible\_non\_5\_ball}; see \cite{BenedettiLutz_LIBRARY} for lists of facets of the complex
      and of its boundary \texttt{contractible\_non\_5\_ball\_boundary}.
      Then \texttt{contractible\_non\_5\_ball} 
      is a non-PL triangulation of a contractible $5$-manifold 
      different from the $5$-ball~$B^5$ with $f$-vector 
      \[f(\mbox{\texttt{contractible\_non\_5\_ball}})=(5013,72300,290944,495912,383136,110880).\]

\begin{thm}\label{thm:non_5_ball}
There is a collapsible $5$-manifold \texttt{contractible\_non\_5\_ball} 
different from the $5$-ball with $f=(5013,72300,290944,495912,383136,110880)$.
\end{thm}

\begin{proof} By construction, the simplicial complex \texttt{contractible\_non\_5\_ball}
triangulates a contractible $5$-manifold different from the $5$-ball. The boundary 
of the resulting $5$-manifold is PL homeomorphic to the double over the homology ball $\Sigma^3\setminus \Delta\times I$, 
where $\Sigma^3$ is the starting Poincar\'e homology sphere and $\Delta$ is any facet. 
In particular, the boundary of \texttt{contractible\_non\_5\_ball} is a combinatorial $4$-manifold
\texttt{contractible\_non\_5\_ball\_boundary} with trivial homology, but is not simply connected
--- it has the binary icosahedral group as its fundamental group, in contrast to the boundary of the $5$-ball.
Moreover, as a regular neighborhood of a tree, it is immediate that the complex is PL collapsible, i.e.,
some suitable subdivision of it is collapsible. Here, we used the prototype implementation \texttt{DiscreteMorse} \cite{Lutz_DiscreteMorse}
in GAP \cite{GAP4} of the random-discrete-Morse approach of~\cite{BenedettiLutz2014}
to directly find an explicit collapsing sequence. 
It took 60:17:33 h:min:sec to build the Hasse diagram of the example
containing $2\cdot 72300 + 3\cdot 290944 + 4\cdot 495912 + 5\cdot 383136 + 6\cdot 110880=5582040$
edges. In a \emph{single} random run of the program
in 21:41:31 h:min:sec, a discrete Morse vector $(1,0,0,0,0,0)$ was achieved.
This proves the collapsibility of the example.  
(Using a reimplementation by Mimi Tsuruga of  the \texttt{DiscreteMorse} program as a 
client (to appear) in the (next release of the) \texttt{polymake} system \cite{GawrilowJoswig2000},
the same result can be achieved in about 9 seconds with the 
\texttt{random-lex-first} and \texttt{random-lex-last} strategies 
and in about 10 minutes with the uniform \texttt{random} strategy,
where the \texttt{random-lex-first} and \texttt{random-lex-last} strategies 
found the optimal discrete Morse vector $(1,0,0,0,0,0)$ in 967 and in 1000
out of 1000 random runs, respectively \cite{JoswigLutzTsuruga2015pre}.)
\end{proof} 

\begin{cor}
The boundary  of \texttt{contractible\_non\_5\_ball}  is a combinatorial $4$-dimensional homology sphere 
 \texttt{contractible\_non\_5\_ball\_boundary}  with $f=(5010,65520,212000,252480,100992)$
that has the binary icosahedral group as its fundamental group.
\end{cor}

(Recently, a first concrete example of a collapsible non-PL triangulation of the $5$-ball for which its boundary is a combinatorial $4$-sphere
was constructed in \cite{AdiprasitoLutzTsuruga2015pre}; see also \cite{Benedetti2012pre} 
for an outline of the construction. This example is based on a triangulation of Mazur's contractible $4$-manifold.)

\section{Asymptotic complicatedness of barycentric subdivisions}
\label{sec:asymptotic}

Mesh refinements are often used in discrete geometry to force nice combinatorial properties, while at the same time 
maintaining the existing ones. The price to pay is of course an increase (linear, polynomial, or even exponential) in the computation.
 
A particularly effective refinement, from this point of view, is the ``barycentric subdivision'', as the following results suggest: 
\begin{compactenum}[(1)]
\item given an arbitrary PL ball, some iterated barycentric  subdivision of it is collapsible  \cite{AdiprasitoBenedetti2011pre};
\item for every PL sphere, some iterated barycentric subdivision of it is polytopal~\cite{AdiprasitoIzmestiev2013pre};
\item for every PL triangulation of any smooth manifold that has a handle decomposition of $c_i$ $i$-handles, some iterated 
         barycentric subdivision of it admits $(c_0, \ldots, c_d)$ as optimal discrete Morse vector~\cite{Benedetti2012pre}.
\end{compactenum}

In this section we focus on the \emph{average} discrete Morse vector, rather than on the smallest one. In \cite[p.~13]{BenedettiLutz2014} 
it was observed experimentally that the (observed) average number of critical cells in the random computation  of discrete Morse vectors 
can decrease for ``complicated'' triangulations (which we experienced for the examples \texttt{trefoil\_bsd}, \texttt{double\_trefoil\_bsd}, 
and \texttt{triple\_trefoil\_bsd}) after performing a single barycentric subdivision. Here we show that this experiment is misleading, 
in the sense that the observed average should rapidly increase after a finite number $\ell$ of barycentric subdivisions.\footnote{Of course, 
observing  this phenomenon experimentally may be difficult, due to (1) computational expense of treating barycentric subdivisions 
after the first or second iteration, (2) the possibility that the probabilities $p_1$, $p_2$, and $p_3$ determined by Lemma \ref{lem:pos} 
below are so small that one needs an extremely large $\ell$ to appreciate the effect.}

We show that after a constant number, say $\alpha$, of barycentric subdivisions (a number larger than~$1$, but universally bounded, 
i.e., independent of the $3$-complex chosen) every $3$-complex $C$ will contain a $2$-dimensional copy of the dunce hat $D$ 
as a subcomplex in any of its subdivided tetrahedra. As a consequence, we will see in Theorem~\ref{thm:exp} that the 
expected number of critical cells in a discrete Morse vector for the $(\ell+\alpha)$-th barycentric subdivision 
of any $3$-dimensional simplicial complex grows exponentially in $\ell$. 

\begin{figure}[tb]
\begin{center}
\begin{postscript}
\psfrag{D}{$D$}
\psfrag{E}{$E$}
\psfrag{C}{$C$}
\psfrag{K}{$K$}
\psfrag{T}{$T$}
\includegraphics[width=5.5cm]{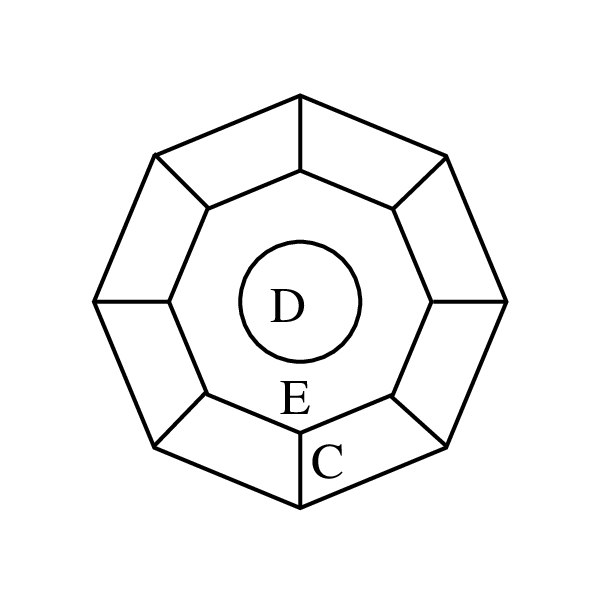}\hspace{7.5mm}\includegraphics[width=5.5cm]{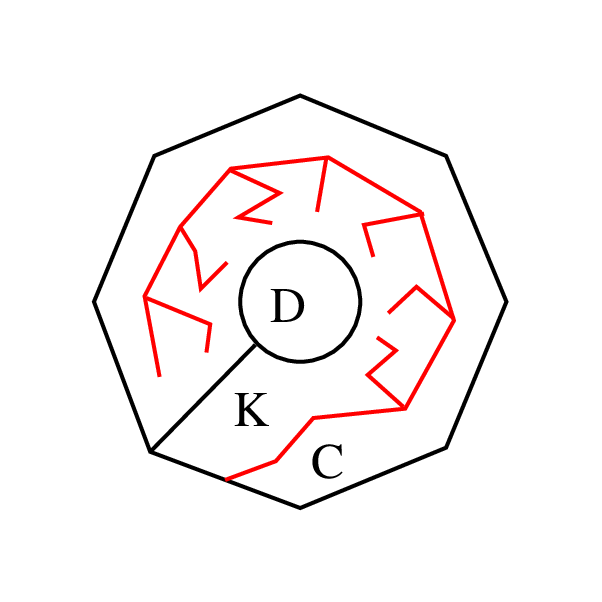}
\end{postscript}
\end{center}
\caption{The dunce hat $D$ sitting inside the $3$-ball $E$, ``shielded'' by $\partial E \times [0,1]$ in the $3$-ball~$C$ (left);
              a collapsing sequence that first empties out the collar $E \times [0,1]$ and then collapses $E$ to $D$, 
              leaving a bridge $K$ (right).}
\label{fig:dunce_in_ball}
\end{figure}

\begin{lem}\label{lem:pos}
There are universal constants $\alpha \in \mathbb{N}$ and $p_1, p_2, p_3 \in (0,1]$ such that
\begin{compactenum}[\rm (i)]
\item the $\alpha$-th iterated barycentric subdivision $\mathrm{sd}^{\alpha}\, T$ of the tetrahedron $T$ contains a (subdivided) 
         dunce hat as a subcomplex; 
\item and if\, $\Sigma\in \mathrm{sd}^{\alpha}\, T$ is any facet intersecting  $\partial\,\mathrm{sd}^{\alpha}\, T$ in dimension $2$, 
         then with probability $p_1>0$, the discrete Morse algorithm \texttt{random} collapses
         $\mathrm{sd}^{\alpha}\, T- \Sigma$ to a $2$-complex containing the subdivided dunce hat and $\partial\,\mathrm{sd}^{\alpha}\, T$;
\item the strategies \texttt{random-lex-first} and \texttt{random-lex-last}
          similarly collapse $\mathrm{sd}^{\alpha}\, T- \Sigma$ to a $2$-complex containing the dunce and hat and $\partial\,\mathrm{sd}^{\alpha}\, T$ 
          with positive probabilities $p_2$ and $p_3$ 
          (under the same restrictions, respectively).
\end{compactenum}
\end{lem}

\begin{proof}
A well known result in PL topology (cf.\ \cite{ZeemanBK}) is that any regular neighborhood of the dunce hat in $\mathbb{R}^3$ is a (PL) $3$-ball; 
moreover, any regular neighborhood contains the original complex in its interior, and collapses in the PL category onto the original complex. 
In other words, there is a simplicial $3$-ball $E$ that collapses onto (a subdivision of) the dunce hat $D$.

Let us ``collar'' $E$, that is, let us consider the $3$-ball
\[ C = E \cup \left (\, \partial E  \times [0,1] \, \right). \]
The ball $C$ is not simplicial anymore, but consists of the tetrahedra of $E$ along with triangular prisms $t\times\, [0,1]$ for each of the triangles $t$ 
of $\partial E$. We triangulate the prisms, e.g., by first subdividing the three square faces of each prism into four triangles each 
(we cone over the boundary $4$-gons) and then inserting a central vertex into each prism and cone over its triangulated boundary. 
We call the resulting triangulated $3$-ball $C'$; it contains in its interior the (triangulated ball $E$ that contains the) copy of the dunce hat $D$, 
``shielded'' by the collar $\partial E \times [0,1] $, where $\partial C'=\partial C$; see also Figure~\ref{fig:dunce_in_ball}.

Because $\partial E$ is two-dimensional, it is endo-collapsible. This implies that for any facet $\Sigma$ of $C'$ not in $E$
(in particular, for a facet $\Sigma$ intersecting $\partial C'$ in a boundary triangle), 
the complex $C' - \Sigma$ collapses onto 
$$F = \partial C' \cup K \cup  E,$$
 where $K$ is some complex of  dimension at most two
(that can be thought of as a bridge between the outer boundary $\partial C'$ of the whole ball $C'$ and the interior ball $E$; 
see Figure~\ref{fig:dunce_in_ball}); and this $F$ collapses in turn to 
$$G=\partial C' \cup K \cup  D.$$ 
In other words, we first empty out the collar $\partial E \times [0,1]$ and then collapse $E$ to $D$,
leaving over only the $2$-complex $K \cup  D$ (and the outside boundary $\partial C'$). This way, we ensure that there are no further collapses
that could collapse away $D$.

Let $T$ be a tetrahedron. Since $C'$ is a PL ball, by simplicial approximation (cf.~\cite{ZeemanBK}) it follows that 
there is some iterated barycentric subdivision $\mathrm{sd}^{\alpha}\, T$ of $T$ that also subdivides $C'$, 
i.e.\ there is a facewise linear homeomorphism $\varphi:C'\mapsto  \mathrm{sd}^{\alpha}\, T$.
With \cite[Thm.~3.5.4]{AdiprasitoBenedetti2011pre}, 
it follows that  $\varphi(C')=\mathrm{sd}^{\alpha}\, T$
collapses to $\varphi(G) \subset \mathrm{sd}^{\alpha}\, T$.

Furthermore, for $\alpha$ large enough, there is a labeling of the vertices of $\mathrm{sd}(\mathrm{sd}^{\alpha-1}\, T)$
such that it collapses to a triangulation of the dunce hat with the lex-first and lex-last orders. 
(As above, there is no loss in assuming that in the collapsing sequence some subdivision of $F$ appears 
as an intermediate step.)

In particular, the probability that any of the algorithms above collapses $\mathrm{sd}^{\alpha}\, T$ to the dunce hat is strictly positive. 
\end{proof}

\begin{rem}\label{rem:tau}
Since according to the proof of Lemma~\ref{lem:pos} we can choose the starting facet $\Sigma$ to intersect  $\partial\,\mathrm{sd}^{\alpha}\, T$ 
in any boundary triangle~$\tau$, we could, instead of $\Sigma$, initially use $\tau$ (as a free face) for an elementary collapse 
$(\tau,\Sigma)$ to enter $\mathrm{sd}^{\alpha}\, T$
as indicated in Figure~\ref{fig:dunce_in_ball} to collapse $\mathrm{sd}^{\alpha}\, T$ to
a $2$-complex containing the subdivided dunce hat and $\partial\,\mathrm{sd}^{\alpha}\, T-\tau.$
\end{rem}

Thus, if we have a $3$-complex $K$ with $N$ facets, by subdividing $K$ barycentrically $\alpha$ times, 
we obtain a $3$-complex that contains in its $2$-skeleton $N$ disjoint copies of the dunce hat (one per each tetrahedron of $K$).
To study how discrete Morse theory behaves on such a complex, we start with a simple lemma. 
In the following, let $|\beta(K)|$ denote the sum of the (unreduced) Betti numbers of a complex $K$.

\begin{lem}\label{lem:embed}
Let $K$ be a $2$-dimensional simplicial complex that contains $r$ disjoint copies of the dunce hat. Then $K$ does not admit 
a discrete Morse function with less than $2r+|\beta(K)|$ critical cells.
\end{lem}

\begin{proof}
We prove the claim by double induction, on $r$ and on the number of faces (``size'') of $K$. The case $r=0$ 
is a straightforward consequence of the Morse inequalities. Thus, let $r\ge 1$ and let us assume the claim is proven 
for all $2$-complexes $K'$ with at most $r$ disjoint embeddings of the dunce hat and less faces than $K$. 
Let $f$ be any optimal discrete Morse function, and let $\sigma$ be the facet at which $f$ attains its maximum. 
There are three cases to consider: 
\begin{compactitem}
\item $\sigma$ is not critical. In this case there is a free face $\rho$ of $\sigma$ such that $f(\rho)=f(\sigma)$. 
         Note that the face $\rho$ cannot belong to a copy of the dunce hat, because in the dunce hat 
         every edge belongs to two or three triangles. The simultaneous removal of both $\rho$ and $\sigma$ 
          is then an elementary collapse; let $K'$ be the complex obtained. $K'$ has the same homology of $K$, 
          and it still has $r$ copies of the dunce hat embedded. The claim follows by applying the inductive assumption to $K'$.
\item $\sigma$ is a critical face and belongs to an embedding $D$ of the dunce hat. 
         In this case, we delete~$\sigma$. By a simple cellular homology argument, one can show that 
         $\beta_1(D-\sigma)=\beta_1(D)+1$ (and $\beta_i(D-\sigma)=\beta_i(D)$ otherwise). 
         Therefore, we obtain that $\beta_1(K-\sigma)=\beta_1(K)+1$ (and $\beta_i(K-\sigma)=\beta_i(K)$ otherwise). 
         Hence $|\beta(K-\sigma)|=\beta(K)|+1$, and $(K-\sigma)$ is a $2$-complex containing at least $r-1$ 
         copies of the dunce hat. The claim follows by applying the inductive assumption to $K'=K-\sigma$.
\item $\sigma$ is a critical face of $K$ not intersecting any copy of the dunce hat.
         We split this case into two subcases. Either the critical cell we are going to remove
         is contained in a homology $2$-cycle or not. In the first case,
         $\beta(K-\sigma)=(\beta_0(K),\beta_1(K),\beta_2(K)-1)$, whereas in the second case,  
         $\beta(K-\sigma)=(\beta_0(K),\beta_1(K)+1,\beta_2(K))$. 
         Together, we have $|\beta(K-\sigma)|\geq |\mbox{$\beta(K)|-1$}$. 
         Again, by applying the inductive assumption to $K'=K-\sigma$,
         the restriction of $f$ to $K-\sigma$ has no less than $2r+|\beta(K)|-1$ 
         critical cells. 
         The claim follows because $f$ has one more critical cell on $K$ than the  restriction of $f$ on $K-\sigma$.\qedhere 
\end{compactitem}
\end{proof}

For the next result, we need to fix some notation first. Given a simplicial complex $K$, we use $f_i(K)$ to denote the number of $i$-dimensional faces of $K$. 
We denote by $X_1(K),\ X_2(K)$ resp.\ $X_3(K)$ the random variables for the number of critical faces of the discrete Morse strategies \texttt{random}, \texttt{random-lex-first} and \texttt{random-lex-last} on $K$, respectively, and let $\mathbb{E}_1(K),\ \mathbb{E}_2(K)$ resp.\ $\mathbb{E}_3(K)$ 
denote the associated expected values.
We write $X_{\ast}(K)$ resp.\ $\mathbb{E}_{\ast}(K)$ when a statement applies to all three strategies.

Recall that we write $f=\varOmega(g)$ resp.\ $f=O(g)$ for real valued functions $f:\mathbb{N}\rightarrow \mathbb{R}$ 
and $g:\mathbb{N}\rightarrow \mathbb{R}$ if there is a $c>0$ such that for all $x\in \mathbb{R}$ large enough, 
$f(x)\ge c g(x)$ resp.\ $f(x)\le c g(x)$. We write $f=\varTheta(g)$ if $f=\varOmega(g)$ and $f=O(g)$.

\begin{thm}\label{thm:exp}
Let $K$ be a simplicial $3$-complex and let $\tilde{p}_i$, $i\in\{1,2,3\}$, be the probability 
(see below), that guarantees to ``find'' the dunce hat~$D$
in the $\alpha$-th barycentric subdivision $\mathrm{sd}^{\alpha}\, T$ of any of the tetrahedra $T$ of~$K$.
Then, for all $\ell\ge 0$ and any $m\geq 0$,
\[\mathbb{P}[X_i(\mathrm{sd}^{\ell+\alpha} K)\le m]\le (1-p_i)^{24^\ell f_3( K)-m},\quad\mbox{for}\quad i\in\{1,2,3\}.\]
Furthermore, $$\log \mathbb{E}_{\ast}(\mathrm{sd}^\ell K)= \varOmega(\ell).$$ 
\end{thm}

\begin{proof}
Let $K$ be a $3$-dimensional simplicial complex $K$ with face vector
$f=(f_0,f_1,f_2,f_3)$ and $k$ be a nonnegative integer, then $f_3(\mathrm{sd}^k K)=24^k f_3(K)$.

If we let, for large enough~$\alpha$, 
the algorithm \texttt{random} (or the algorithms \texttt{random-lex-first} and \texttt{random-lex-last}
after a suitable relabeling of the vertices) run on $\mathrm{sd}^\alpha K$,
then eventually we enter into every subdivided tetrahedron $\mathrm{sd}^{\alpha}\, T$ of any of the initial tetrahedra $T$ of~$K$,
either via a free boundary triangle $\tau$ of $\partial\,\mathrm{sd}^{\alpha}\, T$ or by removing a tetrahedron $\Sigma$
(only if no free faces are available) from $\mathrm{sd}^{\alpha}\, T$. As argued above in the proof of Lemma~\ref{lem:pos}
and in Remark~\ref{rem:tau},
for any choice of a first~$\tau$ there is a sequence of collapses that empties out $\mathrm{sd}^{\alpha}\, T$,
but leaves (a subdivision of) the dunce hat~$D$. If the starting facet $\Sigma$ is picked (randomly)
from (the subdivision of the image of) $C'\setminus E$ we can also empty out  $\mathrm{sd}^{\alpha}\, T$ and leave~$D$.

However, the run of the algorithm \texttt{random} need not target the subdivided tetrahedra one after each other, 
but will jump within $\mathrm{sd}^\alpha K$. Yet, within each $\mathrm{sd}^{\alpha}\, T$
we still do have positive conditional probabilities to get stuck with $D$, depending on 
how many triangles of $\partial\,\mathrm{sd}^{\alpha}\, T$ become free faces at different stages of our run.
There are only finitely many boundary triangles of $\partial\,\mathrm{sd}^{\alpha}\, T$,
and thus there are only finitely many situations to consider for possible collapsing sequences on $\mathrm{sd}^{\alpha}\, T$.
It is important to point out though that whatever goes on beyond the boundary $\partial\,\mathrm{sd}^{\alpha}\, T$ during a run,
this only affects the (conditional) probability of finding a dunce hat $D$ inside $\mathrm{sd}^{\alpha}\, T$,
but does not block the option to \emph{actually} find a dunce hat,
i.e., a run is not independent on the individual $\mathrm{sd}^{\alpha}\, T$, but independent enough
to guarantee our result.
For simplicity, we therefore let $\tilde{p}_i>0$, $i\in\{1,2,3\}$ be respective (lower bounds on the conditional) probabilities
to find a dunce hat.

Let $K$ have an optimal discrete Morse function with $c$ critical cells, 
then for any $\ell\geq 1$ also $\mathrm{sd}^{\ell} K$ has an optimal discrete Morse function 
with at most $c$ critical cells, and
$$\mathbb{P}[X_i(\mathrm{sd}^{\ell+\alpha} K)\le m] \leq\mathbb{P}[X_i(\mathrm{sd}^{\ell+\alpha} K)\le c+m]
\leq \tilde{p}_i^{\,m} \cdot (1-\tilde{p}_i)^{24^\ell f_3( K)-m}  \leq (1-\tilde{p}_i)^{24^\ell f_3( K)-m},$$
for $i\in\{1,2,3\}$, where $m$ dunce hats are picked up (by Lemma~\ref{lem:embed}) 
in the $f_3(\mathrm{sd}^{\ell} K)=24^\ell f_3( K)$ tetrahedra of $\mathrm{sd}^{\ell} K$ 
(which is refined further to $\mathrm{sd}^{\ell+\alpha} K$
to host the dunce hats by Lemma~\ref{lem:pos}), while the other $24^\ell f_3( K)-m$ dunce hats are not picked up. 

For the expected number of critical cells we then have 
$$\mathbb{E}_i(\mathrm{sd}^{\ell+\alpha} K)\geq  c + \tilde{p}_i\cdot f_3(\mathrm{sd}^{\ell} K)=c+\tilde{p}_i\cdot 24^\ell f_3( K),$$
and thus $\log \mathbb{E}_{\ast}(\mathrm{sd}^\ell K)= \varOmega(\ell)$. 
\end{proof}

The bound provided by Theorem \ref{thm:exp} is asymptotically tight:

\begin{cor}\label{cor:highdim}
Let $K$ be any simplicial complex of dimension $d \ge 3$. Then 
\[\log \mathbb{E}_\ast(\mathrm{sd}^{\ell} K)=\varTheta (\ell).\]
\end{cor}

\begin{proof}
First, we have to justify that we can consider any dimension $d \ge 3$.
For this, we start out with the $3$-dimensional ball $E$ of Lemma~\ref{lem:pos}
that contains the dunce hat $D$, we take the cross product
$E^d=E\times I^{d-3}$ to obtain a $d$-ball that we again shield via
\[ C^d = E^d \cup \left (\, \partial E^d \, \times \, [0,1] \, \right)\]
and then place $C^d$ in the interior of a $d$-simplex $\sigma_d$.
Again, there is a universal constant $\alpha_d$ for each dimension $d$
such that $\mathrm{sd}^{\alpha_d} \sigma_d$ contains a triangulated copy of $C^d$,
where $C^d$ collapses onto $D$ (with any of the three strategies) with a positive probability $p$.

By Theorem \ref{thm:exp}, we have $\log \mathbb{E}_{\ast}(\mathrm{sd}^{\ell} K)= \varOmega(\ell)$. 
Furthermore, it is clear that $\log \mathbb{E}_{\ast}(\mathrm{sd}^{\ell} K)$ is bounded above 
by $\log f(\mathrm{sd}^{\ell} K)\le\log( 2^d ((d+1)!)^{\ell} f_d(K))= C\log \ell$, where $f(\cdot)$ denotes the total number 
of faces of a simplicial complex and $C$ is a constant.
\end{proof}

A similar behavior for the randomized algorithms with respect to simplicial polytopes can be observed 
when we increase the dimension of the polytopes concerned.

\begin{thm}
For each simplicial polytope $P$ of dimension $d$, 
\[\mathbb{E}_\ast(P)=  \varOmega(d).\]
\end{thm}

\begin{proof}
The facets of a simplicial $d$-polytope $P$ are $(d-1)$-dimensional simplices that have $d$ vertices each,
and in turn contain any simplicial complex with up to $d$ vertices as a subcomplex.
Let $T$ be the $3$-dimensional tetrahedron, then the $\ell$-th barycentric subdivision
$\mathrm{sd}^\ell T$ consists of $(4!)^{\ell}$ tetrahedra and has at most $4\cdot(4!)^{\ell}$ vertices.
Therefore, if $\ell=\lfloor \log_{4!} \frac{d}{4} \rfloor$, then $P$ contains $\mathrm{sd}^\ell T$.
The assertion now follows from Theorem~\ref{thm:exp} (using appropiate shielding as in the proof of Corollary~\ref{cor:highdim}).
\end{proof}

\noindent
\textbf{Acknowledgments:} We are grateful to the  anonymous referees for valuable comments that greatly helped to improve the presentation of the paper.

\section*{Appendix: Random-lex-first and random-lex-last strategies}
\label{sec:lex_rev_lex}

In \cite{BenedettiLutz2014}, the \texttt{random} discrete Morse strategy was used to search for small discrete Morse vectors for various input.
As pointed out in the previous section,  this random procedure, and also the random versions 
 \texttt{random-lex-first} and \texttt{random-lex-last} of  the deterministic strategies  \texttt{lex-first} and \texttt{lex-last} of  \cite{BenedettiLutz2014},
 respectively,  pick up exponentially many critical cells asymptotically almost surely. 
 Yet, on rather huge examples the heuristics often still are successful in finding good or even optimal discrete Morse vectors --- with the new
 \texttt{random-lex-last} strategy working significantly better than the other implementations; cf.\ Table~\ref{tbl:lex_rev_lex}.
 
As a (non-trivial) testing ground for discrete Morse (heuristical) algorithms,
a library of 45 triangulations was provided in \cite{BenedettiLutz2014}.
For 39 out of the 45 examples, optimal discrete Morse
vectors were found in  \cite{BenedettiLutz2014} with the discrete Morse strategy \texttt{random} 
or with the \texttt{lex-first} and \texttt{lex-last} strategies.
For~4~of the 6 open cases,
\begin{compactitem}
\item \texttt{hyperbolic\_dodecahedral\_space}, (1,4,4,1),
\item \texttt{non\_PL}, (1,0,0,2,2,1),
\item \texttt{S2xpoincare}, (1,2,3,3,2,1),
\item \texttt{contractible\_vertex\_homogeneous}, (1,0,0,4,8,4,0,0,0,0,0,0),
\end{compactitem}
we miss appropriate lower bounds to show optimality of the found discrete Morse vectors.
In the case of the $3$-ball triangulation \texttt{knot}, (1,1,1,0) was the best vector obtained in \cite{BenedettiLutz2014},
while Lewiner in \cite{Lewiner2005} claims to have found (1,0,0,0).
For the example \texttt{triple\_trefoil\_bsd}, (1,1,1,1) was reached in \cite{BenedettiLutz2014}.

\begin{thm}
The triangulated $3$-sphere example \texttt{triple\_trefoil\_bsd} has the perfect discrete Morse vector (1,0,0,1)
in its discrete Morse spectrum.
\end{thm}

We found the discrete Morse vector (1,0,0,1) for the example \texttt{triple\_trefoil\_bsd} 
with both the  \texttt{random-lex-first} and the \texttt{random-lex-last} strategy
in only 7 and 5 out of 10000 runs, respectively,
which provides us an interesting benchmark example for which the optimal discrete Morse vector is hard to find;
see Table~\ref{tbl:discrete_morse_spectra}. 

Table~\ref{tbl:discrete_morse_spectra} displays two further examples,   \texttt{nc\_sphere}  and   \texttt{bing}, 
 for which the optimal vector $(1,0,0,1)$ was found way more often with the \texttt{random-lex-first} and the \texttt{random-lex-last}
 strategies, compared to \texttt{random-discrete-Morse}. Also, for the examples \texttt{poincare} and \texttt{non\_PL}
(and for other examples as well) we see a dramatic simplification of the spectrum when using \texttt{random-lex-last}  (and \texttt{random-lex-first}). 
For the library examples listed in Table~\ref{tbl:lex_rev_lex},
we highlighted (in bold), which of the three random strategies yields the smallest average of critical cells.
In most cases, the \texttt{random-lex-last} strategy scores best --- and even in the case \texttt{triple\_trefoil\_bsd}  
where it did not, it revealed the optimum. The last column of Table~\ref{tbl:lex_rev_lex} lists the best known theoretical 
lower bound for the number of critical cells in a discrete Morse vector. The number is in bold if we can prove 
that it is actually achievable, i.e., there is a discrete Morse function with that many critical cells.

\vfill

\mbox{}

\pagebreak

{\small
\defaultaddspace=.1em

\setlength{\LTleft}{0pt}
\setlength{\LTright}{0pt}
\begin{longtable}{@{}l@{\extracolsep{5mm}}r@{\extracolsep{\fill}}l@{\extracolsep{5mm}}r@{\extracolsep{\fill}}l@{\extracolsep{5mm}}r@{}}
\caption{\protect\parbox[t]{15cm}{Distribution of discrete Morse vectors in 10000 rounds.}}\label{tbl:discrete_morse_spectra}
\\\toprule
 \addlinespace
 \addlinespace
 \addlinespace
 \addlinespace
 \multicolumn{2}{@{}l@{}}{\texttt{random}}      &  \multicolumn{2}{@{}l@{}}{\texttt{random-lex-first}} &  \multicolumn{2}{@{}l@{}}{\texttt{random-lex-last}} \\ \midrule
\endfirsthead
\caption{\protect\parbox[t]{15cm}{Distribution of obtained discrete Morse vectors in 10000 rounds (continued).}}
\\\toprule
 \addlinespace
 \addlinespace
 \addlinespace
 \addlinespace
 \multicolumn{2}{@{}l@{}}{\texttt{random}}      &  \multicolumn{2}{@{}l@{}}{\texttt{random-lex-first}} &  \multicolumn{2}{@{}l@{}}{\texttt{random-lex-last}} \\ \midrule
\endhead
\bottomrule
\endfoot
\addlinespace
 \texttt{poincare}       &       &                     &       &                     &       \\[.5mm]
 \textbf{(1,2,2,1)}:     &  9073 & \textbf{(1,2,2,1)}: &  9992 & \textbf{(1,2,2,1)}: &  9999 \\[-1.1mm]
 $(1,3,3,1)$:            &   864 & $(1,3,3,1)$:        &     8 & $(1,3,3,1)$:        &     1 \\[-1.1mm]
 $(1,4,4,1)$:            &    45 &                     &       &                     &       \\[-1.1mm]
 $(2,4,3,1)$:            &     7 &                     &       &                     &       \\[-1.1mm]
 $(2,3,2,1)$:            &     6 &                     &       &                     &       \\[-1.1mm]
 $(1,5,5,1)$:            &     5 &                     &       &                     &       \\[.5mm]
\addlinespace
 \addlinespace
  \texttt{non\_PL}       &     &                     &       &                     &       \\[.5mm]
   \emph{(1,0,0,2,2,1)}: &  9383 &  \emph{(1,0,0,2,2,1)}: &  9991 &  \emph{(1,0,0,2,2,1)}: &  10000 \\[-1.1mm]
  $(1,0,0,3,3,1)$:        &        441 & $(1,0,0,3,3,1)$:        &         9  \\[-1.1mm]
  $(1,0,1,3,2,1)$:        &        134 & \\[-1.1mm]
  $(1,0,0,4,4,1)$:        &          23 & \\[-1.1mm]
  $(1,0,1,4,3,1)$:        &         12 & \\[-1.1mm]
  $(1,0,2,4,2,1)$:        &           2 & \\[-1.1mm]
  $(1,0,0,5,5,1)$:        &          2 & \\[-1.1mm]
  $(1,0,2,5,3,1)$:        &          1 & \\[-1.1mm]
  $(1,1,2,3,2,1)$:        &          1 & \\[-1.1mm]
  $(1,0,4,6,2,1)$:        &          1 & \\[.5mm]
 \addlinespace
 \addlinespace
\addlinespace
  \texttt{nc\_sphere}    &       &                     &       &                     &       \\[.5mm]
 $(1,1,1,1)$:            &  7902 & $(1,1,1,1)$:        &  7550 & $(1,1,1,1)$:        &  8440 \\[-1.1mm]
 $(1,2,2,1)$:            &  1809 & $(1,2,2,1)$:        &  1660 & $(1,2,2,1)$:        &  1426 \\[-1.1mm]
 $(1,3,3,1)$:            &   234 & \textbf{(1,0,0,1)}: &   720 & \textbf{(1,0,0,1)}: &   110 \\[-1.1mm]
 $(1,4,4,1)$:            &    25 & $(1,3,3,1)$:        &    64 & $(1,3,3,1)$:        &    24 \\[-1.1mm]
 \textbf{(1,0,0,1)}:     &    12 & $(2,3,2,1)$:        &     4 &                     &       \\[-1.1mm]
 $(2,3,2,1)$:            &     9 & $(1,4,4,1)$:        &     2 &                     &       \\[-1.1mm]
 $(1,6,6,1)$:            &     3 &                     &       &                     &       \\[-1.1mm]
 $(2,4,3,1)$:            &     3 &                     &       &                     &       \\[-1.1mm]
 $(2,5,4,1)$:            &     2 &                     &       &                     &       \\[-1.1mm]
 $(1,5,5,1)$:            &     1 &                     &       &                     &       \\
 \addlinespace
\addlinespace
 \addlinespace
  \texttt{bing}          &       &                     &       &                     &       \\[.5mm]
 $(1,1,1,0)$:            &  9764 & $(1,1,1,0)$:        &  9484 & $(1,1,1,0)$:        &  9421 \\[-1.1mm]
 $(1,2,2,0)$:            &   217 & \textbf{(1,0,0,0)}: &   280 & \textbf{(1,0,0,0)}: &   456 \\[-1.1mm]
 \textbf{(1,0,0,0)}:     &     7 & $(1,2,2,0)$:        &   233 & $(1,2,2,0)$:        &   117 \\[-1.1mm]
 $(1,3,3,0)$:            &     6 & $(2,3,2,0)$:        &     2 & $(2,3,2,0)$:        &     4 \\[-1.1mm]
 $(2,3,2,0)$:            &     6 & $(1,3,3,0)$:        &     1 & $(1,3,3,0)$:        &     2 \\[.5mm]
\addlinespace
\addlinespace
 \addlinespace
  \texttt{triple\_trefoil\_bsd}          &     &                     &       &                     &       \\[.5mm]
 $(1,2,2,1)$:              &   4793 &  $(1,2,2,1)$ :  &  5193  &   $(1,2,2,1)$ :  &  4557   \\[-1.1mm]
   $(1,1,1,1)$:             &  3390 &   $(1,3,3,1)$:  &   2531  &  $(1,1,1,1)$:  &   3290  \\[-1.1mm]
   $(1,3,3,1)$:             &  1543 &   $(1,1,1,1)$:  &  1966   &  $(1,3,3,1)$:  &   1790  \\[-1.1mm]
   $(1,4,4,1)$:             &    208 &   $(1,4,4,1)$:  &     278   & $(1,4,4,1)$:  &     305   \\[-1.1mm]
   $(1,5,5,1)$:             &      22 &   $(1,5,5,1)$:  &        19   & $(1,5,5,1)$:  &        31  \\[-1.1mm]
   $(2,3,2,1)$:             &      20 &   \textbf{(1,0,0,1)}:  &          7   & $(2,3,2,1)$:  &         13   \\[-1.1mm]
   $(2,4,3,1)$:             &      17 &   $(2,3,2,1)$:  &          3   &  \textbf{(1,0,0,1)}:  &          5   \\[-1.1mm]
   $(1,6,6,1)$:             &         3 &  $(2,4,3,1)$:  &          3   &  $(2,4,3,1)$:  &          5   \\[-1.1mm]
   $(2,5,4,1)$:             &         3 &                         &               &  $(2,5,4,1)$:  &          2  \\[-1.1mm]
   $(1,8,8,1)$:             &         1 &                         &               &  $(1,6,6,1)$:  &          2  \\[.5mm] 
  \addlinespace
 \addlinespace
 \addlinespace
 \addlinespace
 \addlinespace
\end{longtable}

}

\pagebreak

{\small
\defaultaddspace=.1em

\setlength{\LTleft}{0pt}
\setlength{\LTright}{0pt}
\begin{longtable}{@{}l@{\extracolsep{\fill}}r@{\extracolsep{\fill}}r@{\extracolsep{\fill}}r@{\extracolsep{\fill}}r@{}}
\caption{\protect\parbox[t]{15cm}{Average numbers of critical cells in 10000 runs.}}\label{tbl:lex_rev_lex}
\\[-.7mm]\toprule
 \addlinespace
 \addlinespace
 \addlinespace
 \addlinespace
 Name of example  &    \texttt{random} &  \texttt{r.-lex-first}      & \texttt{r.-lex-last}   & $\ge$     \\[-.7mm] \midrule
\endfirsthead
\caption{\protect\parbox[t]{15cm}{Average numbers of critical cells in 10000 runs (continued).}}
\\[-.7mm]\toprule
 \addlinespace
 \addlinespace
 \addlinespace
 \addlinespace
 Name of example  &   \texttt{random}  &  \texttt{r.-lex-first}     & \texttt{r.-lex-last}  & $\ge$     \\[-.7mm] \midrule
\endhead
\bottomrule
\endfoot
 \addlinespace
 \addlinespace
 \addlinespace
 \addlinespace
  \texttt{d2n12g6}                                                    &  14.0558  & 14.0000   &   \textbf{14.0000} &  \textbf{14} \\[-.5mm]
  \texttt{regular\_2\_21\_23\_1}                                      &  32.1354  &  32.0000  &   \textbf{32.0000}  &  \textbf{32}  \\[-.5mm]
  \texttt{rand2\_n25\_p0.328}                                         &  482.9612  &  \textbf{476.7076}  & 477.7182 & \textbf{476} \\[-.5mm]
  \texttt{trefoil\_arc}                                               &   \textbf{1.0952}  &  \emph{1.6246}  &  \emph{1.2180}  &  \textbf{1}  \\[-.5mm]
  \texttt{trefoil}                                                    &  2.0778  &  \emph{2.2164}  &   \textbf{2.0330}  &   \textbf{2} \\[-.5mm]
  \texttt{double\_trefoil\_arc}                                       &  3.6298  &  \emph{3.7690}  &   \textbf{3.2718}  &  \textbf{3}  \\[-.5mm]
  \texttt{poincare}                                                   &  6.1978  &  6.0016  &  \textbf{6.0002}   &  \textbf{6}  \\[-.5mm]
  \texttt{double\_trefoil}                                            &  3.5356  &  \emph{3.7664}  &   \textbf{3.0778}  &  \textbf{2}  \\[-.5mm]
  \texttt{triple\_trefoil\_arc}                                       &  5.9558  &  5.9376  &  \textbf{5.2700}   &  \textbf{5}  \\[-.5mm]
  \texttt{triple\_trefoil}                                            &  6.0082  &  5.9408  &  \textbf{5.0584}  &  \textbf{4}  \\[-.5mm]
  \texttt{hyperbolic\_dodecahedral\_space}                            & 11.5128   &  10.1582  &  \textbf{10.0906}  &  8  \\[-.5mm]
  \texttt{S\_3\_50\_1033} (random)                                    &  3.2114  &  3.2108   &  \textbf{2.6104}   &  \textbf{2}  \\[-.5mm]
  \texttt{non\_4\_2\_colorable}                                       &  30.6  &  25.3232  &  \textbf{16.8154}  &  \textbf{2}  \\[-.5mm]
  \texttt{Hom\_C5\_K4} ($\mathbb{R}\textbf{P}^3$)                     &   4.0508 &  4.0300  &  \textbf{4.0090}  &  \textbf{4}  \\[-.5mm]
  \texttt{trefoil\_bsd}                                               &  \textbf{2.0202}  &  \emph{2.2388}   &  \emph{2.1052}  &  \textbf{2}  \\[-.5mm]
  \texttt{knot}                                                       &  3.1228  &  3.1262  &  \textbf{3.1034}  &  1  \\[-.5mm]
  \texttt{nc\_sphere}                                                 &  4.4788 &  \textbf{4.2164}  &  4.2728  &  \textbf{2}  \\[-.5mm]
  \texttt{double\_trefoil\_bsd}                                       &  \textbf{3.3426}  & \emph{3.8968}  &  \emph{3.4406}  &  \textbf{2}  \\[-.5mm]
  \texttt{bing}                                                       &  3.0468  &  2.9918  &  \textbf{2.9346}  &  \textbf{1}  \\[-.5mm]
  \texttt{triple\_trefoil\_bsd}                                       &  \textbf{5.7432}  &  \emph{6.2346}  &  \emph{5.8460}   &  \textbf{2}  \\[-.5mm]
  \texttt{Hom\_n9\_655\_compl\_K4} $((S^2\!\times\!S^1)^{\# 13})$  [100 runs]   &  28.84  & 28.46 &  \textbf{28.28} &  \textbf{28}  \\[-.5mm] 
  \texttt{CP2}                                                        &  3.0012  &  \textbf{3.0000}   &  \textbf{3.0000}  &  \textbf{3}  \\[-.5mm]
  \texttt{RP4}                                                        &  5.0492  &  \textbf{5.0000}  &  \textbf{5.0000}  &  \textbf{5}  \\[-.5mm]
  \texttt{K3\_16} (unknown PL type)                                   &  24.8228  &  \textbf{24.0000}  &  \textbf{24.0000}   &  \textbf{24}  \\[-.5mm]
  \texttt{K3\_17} (standard PL type)                                  &  24.8984  &  24.0004  &   \textbf{24.0000}   &  \textbf{24}  \\[-.5mm]
  \texttt{RP4\_K3\_17}                                                &  28.58 &  27.0046  &  \textbf{27.0018}  &   \textbf{27} \\[-.5mm]
  \texttt{RP4\_11S2xS2}                                               &  28.48  &  27.0088  &  \textbf{27.0024}  &   \textbf{27} \\[-.5mm]
  \texttt{Hom\_C6\_compl\_K5\_small} $((S^2\!\times\!S^2)^{\# 29})$   &  63.94  &  60.0284  &  \textbf{60.0066}  &  \textbf{60}  \\[-.5mm]
  \texttt{Hom\_C6\_compl\_K5} $((S^2\!\times\!S^2)^{\# 29})$    [10 runs]     &  83.2  &  65.0  &  \textbf{61.6}  &  \textbf{60}  \\[-.5mm]
  \texttt{SU2\_SO3}                                                   &  4.1354  &  \textbf{4.0000}  &  \textbf{4.0000}   &  \textbf{4}  \\[-.5mm]
  \texttt{non\_PL}                                                    &  6.1328  &  6.0018  &  \textbf{6.0000}   &  2  \\[-.5mm]
  \texttt{RP5\_24}                                                    &  6.1770  &  6.0004  &  \textbf{6.0000}  &  \textbf{6}  \\[-.5mm]
  \texttt{S2xpoincare}                                                &  15.70  &  12.078  &  \textbf{12.028}  &  4  \\[-.5mm]
  \texttt{\_HP2}                                                      &  3.1212  &  \textbf{3.0000}  &   \textbf{3.0000}   &  \textbf{3}  \\[-.5mm]
  \texttt{contractible\_vertex\_homogeneous}  [10 runs]             &  $273.6$    &  $\textbf{17.0}$  &   $\textbf{17.0}$ &   1 \\  

 \addlinespace
 \addlinespace
 \addlinespace
 \addlinespace
 \addlinespace
\end{longtable}

}

\bibliography{./references_random_methods_II}

\end{document}